\documentclass[12pt]{amsart}

\usepackage{amsmath,amssymb}
\usepackage{bm}
\usepackage{graphicx}
\usepackage{ascmac}
\usepackage{amsthm}
\usepackage{amsfonts}
\usepackage{mathrsfs}
\usepackage{mathtools}
\usepackage[all]{xy}
\mathtoolsset{showonlyrefs=true}

\theoremstyle{definition}
\newtheorem{thm}{Theorem}[section]
\newtheorem{dfn}[thm]{Definition}

\newtheorem{cor}[thm]{Corollary}
\newtheorem{prop}[thm]{Proposition}
\newtheorem{lem}[thm]{Lemma}
\newtheorem{rem}[thm]{Remark}
\newtheorem*{nota}{Notations}
\newtheorem{axm}[thm]{Axiom}
\newtheorem*{ootp}{Outline of the proof}
\newtheorem*{oftp}{Organization of this paper}
\newtheorem*{ack}{Acknowledgment}

\numberwithin{thm}{section}

\newcommand{\Zpn}{\mathbb{Z}_{>0}}
\newcommand{\Znn}{\mathbb{Z}_{\geq 0}}
\newcommand{\Z}{\mathbb{Z}}
\newcommand{\Zp}{{\mathbb{Z}}_p}
\newcommand{\Zpt}{{\mathbb{Z}}_{p}^{\times}}
\newcommand{\Zps}{{\mathbb{Z}}_{p^2}}

\newcommand{\Zpb}{\breve{\Z}_p}
\newcommand{\Q}{\mathbb{Q}}
\newcommand{\Qp}{{\Q}_p}
\newcommand{\Qpt}{{\Q}_{p}^{\times}}
\newcommand{\Qps}{{\Q}_{p^2}}
\newcommand{\Qpb}{\breve{\Q}_p}
\newcommand{\Qpbt}{\breve{\Q}_p^{\times}}

\newcommand{\Ql}{{\mathbb{Q}}_{\ell}}
\renewcommand{\O}{\mathcal{O}}

\newcommand{\R}{\mathbb{R}}
\newcommand{\C}{\mathbb{C}}
\newcommand{\Fp}{\mathbb{F}_p}

\newcommand{\Fpbar}{\overline{\mathbb{F}}_p}
\newcommand{\A}{\mathbb{A}}
\newcommand{\Sbar}{\overline{S}}
\newcommand{\D}{\mathbb{D}}
\newcommand{\G}{\mathbb{G}}
\newcommand{\bC}{\mathbf{C}}

\newcommand{\bG}{\mathbf{G}}
\newcommand{\bI}{\mathbf{I}}

\newcommand{\bS}{\mathbb{S}}
\newcommand{\sS}{\mathscr{S}}
\newcommand{\sShat}{\widehat{\sS}}
\newcommand{\sSbar}{\overline{\sS}}
\newcommand{\bV}{\mathbf{V}}
\newcommand{\X}{\mathbb{X}}
\newcommand{\bX}{\mathbf{X}}
\newcommand{\T}{\mathbb{T}}

\newcommand{\gbar}{\overline{g}}

\renewcommand{\hbar}{\overline{h}}
\renewcommand{\H}{\mathbb{H}}
\renewcommand{\L}{\mathbf{L}}
\newcommand{\bL}{\mathbb{L}}

\newcommand{\bi}{\mathbf{i}}

\newcommand{\bx}{\mathbf{x}}
\newcommand{\cA}{\mathcal{A}}
\newcommand{\cB}{\mathcal{B}}
\newcommand{\cF}{\mathcal{F}}
\newcommand{\cG}{\mathcal{G}}

\newcommand{\cZ}{\mathcal{Z}}

\newcommand{\Kb}{\breve{K}}
\newcommand{\bLambda}{\mathbf{\Lambda}}

\DeclareMathOperator{\Hom}{Hom}
\DeclareMathOperator{\End}{End}
\DeclareMathOperator{\id}{id}

\DeclareMathOperator{\Lie}{Lie}

\DeclareMathOperator{\trd}{Trd}

\DeclareMathOperator{\ord}{ord}

\DeclareMathOperator{\Ad}{Ad}
\DeclareMathOperator{\GL}{GL}
\DeclareMathOperator{\GSp}{GSp}
\DeclareMathOperator{\GU}{GU}
\DeclareMathOperator{\GSpin}{GSpin}

\DeclareMathOperator{\SO}{SO}

\DeclareMathOperator{\nilp}{Nilp}

\DeclareMathOperator{\spf}{Spf}

\DeclareMathOperator{\loc}{loc}
\DeclareMathOperator{\sml}{sim}
\DeclareMathOperator{\red}{red}
\DeclareMathOperator{\length}{length}
\DeclareMathOperator{\OGr}{OGr}

\DeclareMathOperator{\si}{ss}

\DeclareMathOperator{\rk}{rank}
\DeclareMathOperator{\RZ}{RZ}
\DeclareMathOperator{\VL}{VL}
\DeclareMathOperator{\BT}{BT}

\DeclareMathOperator{\Irr}{Irr}
\DeclareMathOperator{\Sh}{Sh}
\DeclareMathOperator{\an}{an}
\DeclareMathOperator{\nsm}{nsm}
\DeclareMathOperator{\nfs}{nfs}
\DeclareMathOperator{\der}{der}
\DeclareMathOperator{\bs}{basic}

\DeclareMathOperator{\disc}{disc}

\DeclareMathOperator{\SR}{SR}
\DeclareMathOperator{\NSR}{NSR}

\DeclareMathOperator{\Adm}{Adm}
\DeclareMathOperator{\SpL}{SpL}

\DeclareMathOperator{\Def}{Def}

\DeclareMathOperator{\Fil}{Fil}

\usepackage{geometry}
\title[Spinor Rapoport--Zink spaces]{Rapoport--Zink spaces for spinor groups with special maximal parahoric level structure}
\author[Y.Oki]{Yasuhiro Oki}
\address{Graduate School of Mathematical Sciences, 
the University of Tokyo, 3-8-1 Komaba, Meguro-ku, Tokyo 153-8914, Japan. }
\email{oki@ms.u-tokyo.ac.jp}

\begin{document}
\maketitle

\begin{abstract}
In this article, we give a concrete description of the underlying reduced subscheme of the Rapoport--Zink spaces for spinor similitude groups with special maximal parahoric (and non-hyperspecial) level structure. Moreover, we give two applications of the above result. One of which is describing the structure of the basic loci of mod $p$ reductions of Kisin--Pappas integral models of Shimura varieties for spinor similitude groups with special maximal parahoric level structure at $p$. The other is constructing a variant of the result of He, Li and Zhu, which gives a formula on the intersection multiplicity of the GGP cycles associated codimension $1$ embeddings of Rapoport--Zink spaces for spinor similitude groups. 
\end{abstract}

\tableofcontents

\section{Introduction}\label{intr}

This paper contributes the theory of Kisin--Pappas integral models of Shimura varieties and Rapoport--Zink spaces of Hodge type with parahoric level structure. More precisely, we study the structure of basic loci of Shimura varieties for spinor similitude groups $\GSpin(d,2)$ with certain non-hyperspecial level structure at $p>2$. We also consider the corresponding basic Rapoport--Zink spaces constructed by Hamacher and Kim (\cite{Hamacher2019}). Note that there are such known results in non-hyperspecial cases when $d\leq 4$, since we can rephrase $\GSpin(d,2)$ as certain groups of PEL type by exceptional isomorphisms. For example, the result \cite{Oki2019} for quaternionic unitary similitude groups of degree $2$ can be regarded as that for $\GSpin(3,2)$. See Section \ref{mtsh} for more details. However, there was no such a result for $d\geq 5$, in particular they are not of PEL type. 

Our study is motivated by the theory of arithmetic intersection on Shimura varieties and Rapoport--Zink spaces. Here we recall the theory of GGP cycles on Rapoport--Zink spaces (``GGP'' means Gan, Gross and Prasad). In \cite{Li2018}, Li and Zhu defined the GGP cycles for codimension $1$ embeddings of basic Rapoport--Zink spaces for $\GSpin(d,2)$ with hyperspecial level. Moreover, they computed the intersection multiplicities of the GGP cycles in a special case. This calculation was improved by He, Li and Zhu (\cite{He2019}), which successes in generalizing the result of \cite{Li2018}. Their results are expected to be useful for a similar consideration on Shimura varieties. In this paper, we also prove a variant of the main result of \cite{He2019}. See Section \ref{mtai} for more details. 

Note that the researches \cite{Li2018} and \cite{He2019} are inspired by the same study as the case of unitary similitude groups $\GU(1,n-1)$, in which the given imaginary quadratic field is inert at $p>2$. In \cite{Zhang2012}, Zhang conjectured a presentation of intersection multiplicities of the GGP cycles on the basic Rapoport--Zink spaces with hyperspecial level by means of the first derivatives of orbital integrals. Here the GGP cycles are defined as the same manner as the $\GSpin(d,2)$ case. He calls this conjecture as the arithmetic fundamental lemma. At the same time, he also proved that the arithmetic fundamental lemma gives us the similar result on Shimura varieties, which is known as the arithmetic Gan--Gross--Prasad conjecture. The arithmetic fundamental lemma was recently proved by He, Li and Zhu in the minuscule case (\cite{He2019}), and by Zhang when $p\geq n$ (\cite{Zhang2019b}). 

We explain the method of our study on basic loci and Rapoport--Zink spaces. We follow the strategy of the Vollaard's paper \cite{Vollaard2010}. In fact, we use the $p$-adic uniformization theorem which was proved in \cite{Oki2020b}, and reduce the study on basic loci to those on the basic Rapoport--Zink spaces. The Rapoport--Zink spaces are given by certain closed formal subschemes of deformation spaces of principally polarized $p$-divisible groups. See Section \ref{rzsp} for more details. To give concrete descriptions of the structure of them, we use the Dieudonn{\'e} theory, and construct the Bruhat--Tits stratification. It asserts that we can stratify the underlying topological spaces of the Rapoport--Zink spaces by families of classical Deligne--Lusztig varieties indexed by vertices of the Bruhat--Tits buildings of the groups which acts on them. For the unramified $\GU(1,n-1)$ with hyperspecial level, it is known by Vollaard for $n=3$ (\cite{Vollaard2010}), and by Vollaard and Wedhorn for arbitrary $n$ (\cite{Vollaard2011}). Moreover, Howard and Pappas gave the Bruhat--Tits stratification in the case for $\GSpin(d,2)$ with hyperspecial level (\cite{Howard2017}). For results on other groups, see the introduction of \cite{Oki2019} for example. 

\subsection{Rapoport--Zink spaces for spinor similitude groups}\label{mtlr}

Let $\Qp$ be the field of $p$-adic numbers where $p>2$, and $\Zp$ the integer ring of $\Qp$. Moreover, we write $\Qpb$ for the $p$-adic completion of the maximal unramified extension of $\Qp$, and for $\sigma$ the Frobenius of $\Qpb$. 

Let $(V,Q)$ be a quadratic space over $\Qp$ of dimension $n\geq 3$ which admits an almost self-dual lattice, that is, there is a $\Zp$-lattice $\bLambda$ in $V$ such that $\bLambda \subset \bLambda^{\vee}$ and $\length_{\Zp}(\bLambda^{\vee}/\bLambda)=1$. Here $\bLambda^{\vee}$ is the dual lattice of $\bLambda$ in $V$ with respect to $Q$. Let us fix such a lattice, and put $G:=\GSpin(V)$, the spinor similitude group of $V$, and denote by $K_p$ the stabilizer of $\bLambda$ in $G(\Qp)$. Then, for a minuscule cocharacter $\mu$ and a \emph{basic} $b\in G(\Qpb)$ whose $\sigma$-conjugacy class is contained in $B(G,\{\mu\})$ (and a certain embedding of $G$ into $\GSp(2^{n+1})$), we can consider the Rapoport--Zink space $\RZ_{K_p}$, a formal scheme which is locally formally of finite type over $\spf \Zpb$. See Section \ref{rzsp} for more details. Our first goal in this article is describing completely the singularity and the structure of the underlying topological space of $\RZ_{K_p}$. 

First, let us state the results on the structure of $\RZ_{K_p}^{\red}$, the underlying topological space of $\RZ_{K_p}$. There is a quadratic space $\L_{0}$ of dimension $n$ over $\Qp$ such that $\GSpin(\L_{0})(\Qp)$ acts on $\RZ_{K_p}$. We say that a lattice $\Lambda$ in $\L_{0}$ is a \emph{vertex lattice} if $p\Lambda \subset \Lambda^{\vee}\subset \Lambda$. We denote by $\VL(\L_{0})$ the set of vertex lattices in $\L_{0}$. For $\Lambda \in \VL(\L_0)$, we define the \emph{type} of $\Lambda$ as $t(\Lambda):=\length_{\Zp}(\Lambda^{\vee}/\Lambda)$. Then we can prove the equality
\begin{equation*}
\{t(\Lambda)\in \Znn \mid \Lambda \in \VL(\L_{0})\}=\{1,3,\ldots, t_{\max}\}, 
\end{equation*}
where
\begin{equation*}
t_{\max}:=
\begin{cases}
n-1&\text{if $n$ is even},\\
n-2&\text{if $n$ is odd and $\varepsilon(V)=(p,-1)_{p}^{(n-1)/2}$},\\
n&\text{if $n$ is odd and $\varepsilon(V)\neq (p,-1)_{p}^{(n-1)/2}$}. 
\end{cases}
\end{equation*}
Here $\varepsilon(V)$ is the Hasse invariant of $V$, and $(\,,\,)_{p}$ is the Hilbert symbol for $\Qp$. See Corollary \ref{vltn}. We denote by $\VL(\L_0,t)$ the set of vertex lattices in $\L_0$ of type $t$. 

For $\Lambda \in \VL(\L_0)$, we associate a locally closed subscheme $\BT_{K_p,\Lambda}$ of $\RZ_{K_p}$. See Definition \ref{btst}. Moreover, put $\BT_{K_p,\Lambda}^{(0)}:=\BT_{K_p,\Lambda}\cap \RZ_{K_p}^{(0)}$. 

\begin{thm}\label{thl1}
\begin{enumerate}
\item (Theorem \ref{jcis}) \emph{There is a decomposition into connected components
\begin{equation*}
\RZ_{K_p}=\coprod_{i\in \Z}\RZ_{K_p}^{(i)}. 
\end{equation*}
All $\RZ_{K_p}^{(i)}$ are isomorphic to each other. }
\item (Theorem \ref{btnh} (i)) \emph{There is a locally closed stratification
\begin{equation*}
\RZ_{K_p}^{(0),\red}=\coprod_{\Lambda \in \VL(\L_{0})}\BT_{K_p,\Lambda}^{(0)}. 
\end{equation*}}
\item (Theorem \ref{btnh} (ii)) \emph{For $\Lambda \in \VL(\L_{0})$, let $\RZ_{K_p,\Lambda}^{(0)}$ be the scheme-theoretic closure of $\BT_{K_p,\Lambda}^{(0)}$ in $\RZ_{K_{p}}^{(0)}$. Then there is an equality
\begin{equation*}
\RZ_{K_p,\Lambda}^{(0),\red}=\coprod_{\Lambda' \subset \Lambda}\BT_{K_p,\Lambda'}^{(0)}. 
\end{equation*}}
\item (Theorem \ref{dlnh}) \emph{For $\Lambda \in \VL(\L_{0})$, the $\Fpbar$-scheme $\BT_{K_p,\Lambda}^{(0)}$ is isomorphic to the Deligne--Lusztig variety for $\SO_{(t(\Lambda)-1)/2}$ associated with a Coxeter element. Moreover, it is of dimension $(t(\Lambda)-1)/2$. }
\end{enumerate}
\end{thm}

In particular, $\RZ_{K_{p}}^{(0),\red}$ is pure of dimension $(t_{\max}-1)/2$. Moreover, all irreducible components are smooth over $\Fpbar$, and isomorphic to each other. 

\begin{rem}
Theorem \ref{thl1} (ii)--(iv) are predicted in \cite[\S 5]{Gortz2020}. In the following, we give a correspondence between the conditions on $V$ and Enhanced Tits data. 
\begin{table}[h]
\begin{tabular}{c|c}
\hline
Conditions on $V$&Enhanced Tits datum\\ \hline
$2\mid n$&$(C\text{-}B_{(n-2)/2},\omega_{1}^{\vee},\widetilde{\bS}-\{s_0\})$\\
$2\nmid n$, $\varepsilon(V)=(p,-1)^{(n-1)/2}$&$(B_{(n-1)/2},\omega_1^{\vee},\widetilde{\bS}-\{s_n\})$\\
$2\nmid n$, $\varepsilon(V)\neq (p,-1)^{(n-1)/2}$&$({}^2B_{(n-1)/2},\omega_1^{\vee},\widetilde{\bS}-\{s_n\})$\\ \hline
\end{tabular}
\end{table}
\end{rem}

Second, let us state the result on the singularity of $\RZ_{K_{p}}$. For $y\in \RZ_{K_{p}}$, we denote by $\widehat{\RZ}_{K_p,y}$ the complete neighborhood of $y$ in $\RZ_{K_{p}}$. On the other hand, set
\begin{equation*}
R_n:=\Zpb[[t_1,\ldots,t_{n-1}]]/(p+\sum_{i=1}^{n-1}t_it_{n-i}). 
\end{equation*}

\begin{thm}\label{thl2}
\begin{enumerate}
\item (Theorem \ref{rzrg} (i)) \emph{The formal scheme $\RZ_{K_{p}}$ is regular of dimension $n-1$, and is flat over $\spf \Zpb$. }
\item (Theorem \ref{sgt1}) \emph{Let $\RZ_{K_{p}}^{\nfs}$ be the non-formally smooth locus of $\RZ_{K_{p}}$. Then there is an equality
\begin{equation*}
\RZ_{K_{p}}^{\nfs}=\coprod_{\Lambda \in \VL(\L_{0},1)}\RZ_{K_p,\Lambda}=\coprod_{\Lambda \in \VL(\L_{0},1)}\BT_{K_p,\Lambda}. 
\end{equation*}}
\item (Theorem \ref{rzrg} (ii)) \emph{For $y\in \RZ_{K_{p}}^{\nfs}$, there is an isomorphism $\widehat{\RZ}_{K_p,y}\cong \spf R_{n}$. }
\end{enumerate}
\end{thm}

\subsection{Basic loci of Shimura varieties}\label{mtsh}

Let $\bV$ be a quadratic space over $\Q$ of signature $(d,2)$ where $d\in \Zpn$, and put $\bG:=\GSpin(\bV)$. Moreover, let $\bX$ be the moduli space of oriented negative planes in $\bV \otimes_{\Q}\R$. Then $(\bG,\bX)$ is a Shimura datum whose reflex field of $(\bG,\bX)$ is $\Q$. Moreover, $\bG_{\Qp}$ splits over a tamely ramified extension and $p\nmid \pi_1(\bG_{\Qp}^{\der})$. Hence we can consider Kisin--Pappas integral models of Shimura varieties for $(\bG,\bX)$. 

Here we assume that $V:=\bV\otimes_{\Q}\Qp$ satisfies the same assumption as in Section \ref{mtlr}. Let $K_p\subset \bG(\Qp)$ the subgroup as in Section \ref{mtlr}. Take a neat compact open subgroup $K^p$ of $\bG(\A_f^p)$, and put $K:=K^pK_p$. Consider the Kisin--Pappas integral model $\sS_{K}$ of the Shimura variety $\Sh_{K}(\bG,\bX)$, which is defined over $\Z_{(p)}$. Let $\sSbar_{K}$ be the geometric special fiber of $\sS_{K}$, and $\sSbar_{K}^{\bs}$ the basic locus of $\sSbar_{K}$. Note that $\sSbar_{K}^{\bs}$ is a non-empty closed subset of $\sSbar_{K}$ by \cite[Theorem 1.3.13]{Kisin}. 
We denote by $(\sShat_{K,\Zpb})_{/\sSbar_{K}^{\bs}}$ the formal completion of $\sS_{K,\Zpb}:=\sS_{K}\otimes_{\Z_{(p)}}\Zpb$ along $\sSbar_{K}^{\bs}$. 

Let $\bI$ be an inner form of $\Q$ which is anisotropic modulo center over $\R$ satisfying
\begin{equation*}
\bI \otimes_{\Q}\Ql \cong
\begin{cases}
\bG \otimes_{\Q}\Ql &\text{if }\ell \neq p,\\
\GSpin(\L_0) &\text{if }\ell=p. 
\end{cases}
\end{equation*}

\begin{thm}\label{thg1}(Theorem \ref{ccl1})
\emph{
\begin{enumerate}
\item The basic locus $\sSbar_{K}^{\bs}$ is pure of dimension $(t_{\max}-1)/2$. Every irreducible component of $\sSbar_{K}^{\bs}$ is birational to a smooth projective variety. 
\item Let $\pi_0(\sSbar_{K}^{\bs})$ be the set of connected components of $\sSbar_{K}^{\bs}$. Then there is a bijection
\begin{equation*}
\pi_0(\sSbar_{K}^{\bs})\cong \bI(\Q)\backslash (\Z \times \bG(\A_f^p)/K^p). 
\end{equation*}
\item Let $\Irr(\sSbar_{K}^{\bs})$ be the set of irreducible components of $\sSbar_{K}^{\bs}$. Then there is a bijection
\begin{equation*}
\Irr(\sSbar_{K}^{\bs})\cong \bI(\Q)\backslash (\VL(\L_0,t_{\max})\times \bG(\A_f^p)/K^p). 
\end{equation*}
\end{enumerate}}
\end{thm}

\begin{rem}
We can regard some results on the supersingular loci of Shimura varieties and Rapoport--Zink spaces of PEL type as special cases of Theorems \ref{thl1} and \ref{thg1} respectively, by using exceptional isomorphisms. In the following, for some results on those of PEL type, we attach conditions on $V$ as explained above. 
\begin{table}[h]
\begin{tabular}{c|c}
\hline
Results on PEL type &Conditions on $V$\\ \hline
$Y_0(pN)^{\si}$ ($p\nmid N$) is a finite set&$n=3$, $\varepsilon(V)=(p,-1)_{p}$ \\
\cite[\S 2]{Kudla2000a}&$n=3$, $\varepsilon(V)\neq (p,-1)_{p}$ \\
\cite{Bachmat1999} in the ramified case&$n=4$\\
\cite{Yu2011}, \cite[\S 7]{Wang2020} &$n=5$, $\varepsilon(V)=1$ \\ 
\cite{Oki2019}, \cite[\S 5]{Wang2020}&$n=5$, $\varepsilon(V)\neq 1$ \\
\cite{Oki2020a}&$n=6$ \\ \hline
\end{tabular}
\end{table}

Here $n:=d+2$, and $Y_0(pN)^{\si}$ is the supersingular locus of the modular curve of $\Gamma_0(pN)$-level. 
\end{rem}

We give an explicit description of the non-smooth locus $\sS_{K,\Zpb}^{\nsm}$ of $\sS_{K}\otimes_{\Z_{(p)}}\Zpb$. It is known that $\sS_{K,\Zpb}^{\nsm}$ is a finite set contained in the special fiber by \cite[Proposition 12.6, Proposition 12.7]{He2020a}. 

\begin{thm}\label{thg2}
\begin{enumerate}
\item (Theorem \ref{sgbs} (iii)) \emph{The scheme $\sS_{K,\Zpb}^{\nsm}$ is contained in $\sSbar_{K}^{\bs}$. }
\item (Theorem \ref{sgrs}) \emph{There is a bijection
\begin{equation*}
\sS_{K,\Zpb}^{\nsm}\cong \bI(\Q)\backslash (\VL(\L_0,1)\times \bG(\A_f^p)/K^p). 
\end{equation*}}
\end{enumerate}
\end{thm}

\subsection{Application to arithmetic intersection}\label{mtai}

Let $(V^{\sharp},Q^{\sharp})$ be a quadratic space of dimension $n+1$ satisfying the following: 
\begin{itemize}
\item $(V^{\sharp},Q^{\sharp})$ admits a self-dual $\Zp$-lattice,
\item there is an isomorphism $V^{\sharp}\cong V\oplus \Qp x_0$, where $x_0^2\in p\Zpt$. Here the right-hand side is the direct sum of the quadratic spaces. 
\end{itemize}
Fix a self-dual lattice $\bLambda^{\sharp}$ in $V^{\sharp}$, and let $K_p$ be the stabilizer of $L$. Note that $\bLambda:=\bLambda^{\sharp}\cap V$ is an almost self-dual lattice in $V$. Let $K_p^{\sharp}$ be the stabilizer of $L^{\sharp}$ in $G^{\sharp}(\Qp)$, and put $K_p=K_{p}^{\sharp}\cap G(\Qp)$. Then there is the basic Rapoport--Zink space $\RZ_{K_p^{\sharp}}$ for $G^{\sharp}$ with level $K_p^{\sharp}$. Moreover, if we consider the quadratic space $\L_0^{\sharp}:=\L_0\oplus \Qp x_0$ of dimension $n+1$ over $\Qp$, then $\GSpin(\L_0^{\sharp})$ acts on $\RZ_{K_p^{\sharp}}$. We define the notion of vertex lattices in $\L_0^{\sharp}$ as the same manner as the case for $\L_0$. We denote by $\VL(\L_0^{\sharp})$ the set of vertex lattices in $\L_0^{\sharp}$. 

On the other hand, we can construct a closed immersion
\begin{equation*}
\delta_{G,G^{\sharp}}\colon \RZ_{K_p}\hookrightarrow \RZ_{K_p^{\sharp}}, 
\end{equation*}
which is compatible with the actions of $\GSpin(\L_0)(\Qp)$ and $\GSpin(\L_0^{\sharp})(\Qp)$. See Proposition \ref{embd}. This induces a closed immersion
\begin{equation*}
p^{\Z}\backslash \RZ_{K_p}\hookrightarrow p^{\Z}\backslash (\RZ_{K_p}\times_{\spf \Zpb} \RZ_{K_p^{\sharp}}). 
\end{equation*}
Now we write $\Delta_{G,G^{\sharp}}$ for its image, which is said to be the \emph{GGP cycle}. 

Let $g\in \GSpin(\L_0^{\sharp})(\Qp)$. We consider the intersection multiplicity
\begin{equation*}
\langle \Delta_{G,G^{\sharp}},g\Delta_{G,G^{\sharp}}\rangle:=\chi(\O_{\Delta_{G,G^{\sharp}}}\otimes^{\bL}\O_{g\Delta_{G,G^{\sharp}}}). 
\end{equation*}

We say that $g\in \GSpin(\L_0^{\sharp})(\Qp)$ is \emph{regular semisimple minuscule} if the $\Zp$-submodule
\begin{equation*}
L_{x_0}(g):=\Zp x_0\oplus \cdots \Zp(g^{n}x_0)
\end{equation*}
of $\L_0^{\sharp}$ satisfies $L_{x_0}(g)\in \VL(\L_0^{\sharp})$. Here $L_{x_0}(g)^{\cup}$ denotes the dual lattice of $L_{x_0}(g)$ in $\L_0^{\sharp}$. In this case, $\langle \Delta_{G,G^{\sharp}},g\Delta_{G,G^{\sharp}}\rangle$ becomes an integer. See Proposition \ref{trvn} (ii). 

Under the condition that $g$ is regular semisimple minuscule, we prepare some notations which will be needed in the main result. Let $P_g \in \Fp[T]$ be the characteristic polynomial of the action of $g$ on the quadratic space $L_{x_0}(g)^{\cup}/L_{x_0}(g)$ over $\Fp$. We say that the irreducible monic factor $R$ of $P_g$ is \emph{self-reciprocal} if $R(T)=T^{\deg(R)}R(T^{-1})$. We denote by $\SR(P_g)$ and $\NSR(P_g)$ the sets of irreducible monic factors of $P_g$ that are self-reciprocal and not self-reciprocal respectively. Moreover, let $\NSR(P_g)/\!\sim$ be the quotient of $\NSR(P_g)$ under the equivalence relation $R(T)\sim T^{\deg(R)}R(T^{-1})$. We write $[R]\in \NSR(P_g)$ for the equivalence class containing $R\in \NSR(P_g)$. On the other hand, for $R\in \SR(P_g)\sqcup \NSR(P_g)$, we write the multiplicity of $R$ in $P_g$as $m_{R}(P_g)$. Note that $m_{R}(P_g)$ only depends on $[R]$ if $R\in \NSR(P_g)$. 

\begin{thm}\label{thi1} (Theorem \ref{itar} (i), Theorem \ref{thai})
\emph{Assume that $g\in \GSpin(\L_0^{\sharp})(\Qp)$ is regular semisimple minuscule and $\RZ_{K_{p}^{\sharp}}^{g}$ is non-empty. 
\begin{enumerate}
\item We have $\Delta_{G,G^{\sharp}} \cap g\Delta_{G,G^{\sharp}} \neq \emptyset$ if and only if there is a unique $Q_g\in \SR(P_g)$ such that $m_{Q_g}(P_g)$ is odd. 
\item If the equivalent conditions in (i) hold, then we have
\begin{equation*}
\langle \Delta_{G,G^{\sharp}},g\Delta_{G,G^{\sharp}} \rangle =\deg(Q_g)\cdot \frac{m_{Q_g}(P_g)+1}{2}\prod_{[R]\in \NSR(P_g)/\sim}(1+m_{R}(P_g)). 
\end{equation*}
\end{enumerate}}
\end{thm}

\begin{ootp}
We give an outline of the proofs of the main results. The proof of Theorem \ref{thl1} is an analogue of \cite{Oki2019}. More precisely, we construct a closed immersion $\delta_{G,G^{\sharp}}$ as appeared in Section \ref{mtai}, and regard $\RZ_{K_p}$ as a closed formal subscheme of $\RZ_{K_p^{\sharp}}$. Furthermore, the closed immersion $\delta_{G,G^{\sharp}}$ and the natural inclusion $\L_0 \hookrightarrow \L_0^{\sharp}$ induces an injection $\VL(\L_0)\hookrightarrow \VL(\L_0^{\sharp})$, which is denoted by $\Lambda \mapsto \Lambda^{\sharp}$. On the other hand, Howard and Pappas attached a closed formal subscheme $\RZ_{K_p^{\sharp},\Lambda'}$ of $\RZ_{K_p^{\sharp}}$ for any $\Lambda' \in \VL(\L_0)$, and proved that $\{\RZ_{K_p^{\sharp},\Lambda'}\}_{\Lambda' \in \VL(\L_0^{\sharp})}$ covers $\RZ_{K_p^{\sharp}}^{\red}$ (\cite{Howard2017}). Under these notations, let $\RZ_{K_p,\Lambda}:=\RZ_{K_p^{\sharp},\Lambda^{\sharp}}\cap \RZ_{K_p}$ for $\Lambda \in \VL(\L_0)$. Then we can prove $\RZ_{K_p,\Lambda}=\RZ_{K_p^{\sharp},\Lambda^{\sharp}}$ by an analysis of certain lattices in $\L_0^{\sharp}$ or its base change, and the locally closed stratification attached to $\{\RZ_{K_p,\Lambda}\}_{\Lambda \in \VL(\L_0)}$ satisfies the desired properties. 

The proof of Theorem \ref{thg1} is based on \cite{Vollaard2010} as explained above, that is, we apply the $p$-adic uniformization theorem on $\sSbar_{K}^{\bs}$ (Theorem \ref{unif}) and use Theorem \ref{thl1}. 

Theorems \ref{thl2} and \ref{thg2} are consequences of the study of local models by \cite[\S 12]{He2020a} and the local model diagram. The local model diagram is \cite[Theorem 4.2.7]{Kisin2018} in the case for $\sS_{K}$, and is \cite[Theorem 5.2]{Oki2020b} in the case for $\RZ_{K_p}$. 

Finally, we give a strategy to the proof of Theorem \ref{thi1}. We use the same arguments as \cite{Li2018} and \cite{He2019} to prove (i) and (ii) respectively. Note that both \cite{Li2018} and \cite{He2019} are studied on the same situation except that $K_p$ is also hyperspecial. However, we can see that their works can be applied to our case. 
\end{ootp}

\begin{oftp}
In Section \ref{prlm}, we recall the theory of local models and deformations of $p$-divisible groups with additional structure in the case of spinor similitude groups. In Section \ref{kpim} and \ref{rzsp}, we give definitions of Kisin--Pappas integral models and Rapoport--Zink spaces for spinor similitude groups respectively. In Section \ref{dlso}, we recall the results in \cite{Oki2019} on the Deligne--Lusztig varieties for odd special orthogonal groups. In Section \ref{rzgs}, we introduce some notions and collect the results on $\RZ_{K_p^{\sharp}}$ as in \cite{Howard2017}. In Section \ref{rznh}, we construct the Bruhat--Tits stratification of $\RZ_{K_p}$ by using the theory in Section \ref{rzgs}, and give proofs of the main results in Section \ref{mtlr}. In Section \ref{bslc}, we recall the $p$-adic uniformization theorem in our case, and prove the main throerms in Section \ref{mtsh}. Finally, in Section \ref{aris}, we introduce the notion of GGP cycles with respect to $\delta_{G,G^{\sharp}}$, and calculate their intersection multiplicities in the regular semisimple minuscule case. 
\end{oftp}

\begin{ack}
I would like to thank my advisor Yoichi Mieda for his constant support and encouragement.

This work was carried out with the support from the Program for Leading Graduate Schools, MEXT, Japan. This work was also supported by the JSPS Research Fellowship for Young Scientists and KAKENHI Grant Number 19J21728.
\end{ack}

\begin{nota}
\begin{itemize}
\item \textbf{$p$-adic fields and their extensions. } We denote by $\Gamma$ the absolute Galois group of $\Qp$, and by $I$ its inertia subgroup. We normalize the $p$-adic valuation $\ord_{p}\colon \Qpb^{\times}\rightarrow \Z$ on $\Qpb^{\times}$ so that $\ord_{p}(p)=1$. 
\item \textbf{Bruhat--Tits buildings and parahoric group schemes. }For a reductive connected group $G$ over $\Qp$, we write $\cB(G,\Qp)$ for the (extended) Bruhat--Tits building of $G$. For $\bx \in \cB(G,\Qp)$, let $\cG_{\bx}$ be the smooth affine group scheme over $\Zp$ corresponding to $\bx$, and denote by $\cG_{\bx}^{\circ}$ the identity component of $\cG_{\bx}$. 
\item \textbf{Total tensor algebra. }Let $R$ be a ring, and $M$ an $R$-module. We denote by $M^{*}$ the dual module of $M$. Moreover, we define $M^{\otimes}$ as the direct sum of all $R$-modules that are obtained by finite combinations of duals, tensor products, symmetric products and alternating products of $M$. 
\item \textbf{Schemes and formal schemes. }For a formal scheme $\mathscr{X}$ over $\spf \Zpb$, we denote by $\mathscr{X}^{\red}$ the underlying reduced subscheme of $\mathscr{X}$. Furthermore, for $x\in \mathscr{X}$, let $\widehat{\mathscr{X}}_x$ be the complete neighborhood of $x$. On the other hand, for a $\Fp$-scheme $X$, we write the perfection of $X$ as $X^{p^{-\infty}}$. 
\end{itemize}
\end{nota}

\section{Preliminalies}\label{prlm}

\subsection{Quadratic spaces and spinor similitude groups}\label{quad}

Here we recall the theory of quadratic spaces and related algebraic groups. Let $k$ be a field, and $(U,Q)$ a non-degenerate quadratic space of dimension $n$ over $k$. We define a symmetric bilinear form $[\,,\,]$ to be
\begin{equation*}
U\times U\rightarrow k;(v,v')\mapsto Q(v+v')-Q(v)-Q(v'). 
\end{equation*}

We say that $U$ is \emph{isotropic} if there is $v\in U\setminus \{0\}$ satisfying $Q(v)=0$. Otherwise, we say that $U$ is \emph{anisotropic}. 

In the following, we give some specific quadratic spaces. 
\begin{enumerate}
\item For $a\in k^{\times}$, let 
\begin{equation*}
Q_a\colon k \rightarrow k; u \mapsto au^2. 
\end{equation*}
Then $(k,Q_a)$ is an anisotropic quadratic space of dimension $1$ over $k$. We write $l_{a}$ for this quadratic space in the sequel. 
\item Consider the $k$-vector space $k^{\oplus 2}$ of dimensional $2$, and denote by $e_1,e_2$ the standard basis. Set 
\begin{equation*}
Q_0\colon k^{\oplus 2}\rightarrow k;a_1e_1+a_2e_2\mapsto a_1a_2. 
\end{equation*}
Then $(k^{\oplus 2},Q_0)$ is an isotropic quadratic space over $k$. We write this quadratic space as $\H_{k}$ in the sequel. 
\end{enumerate}

Let $C(U)$ be the Clifford algebra of $U$, that is, the quotient of the tensor algebra of $U$ by the ideal generated by $v\otimes v'+v'\otimes v-[v,v']$ for all $v,v'\in U$. Then $C(U)$ is a semisimple algebra over $k$ of dimension $2^{\dim_{k}(U)}$. The center $Z(U)$ of $C(U)$ is $k$ if $2\mid n$, and a quadratic {\'e}tale $k$-algebra if $2\nmid n$. 

The Clifford algebra $C(U)$ of $U$ has a $\Z/2$-grading
\begin{equation*}
C(U)=C^{+}(U)\oplus C^{-}(U),
\end{equation*}
where $C^{\pm}(U)$ is the $\pm 1$-eigenspace of the $k$-linear map induced by $v\mapsto -v$ for any $v\in U$. Note that $C^{+}(U)$ is a subalgebra of $C(U)$, and we have $\dim_{k}(C^{+}(V))=\dim_{k}(C^{-}(V))$. We regard $U$ as a subspace of $C^{-}(U)$ by the natural way. On the other hand, we have an involution $v\mapsto v^{\dagger}$ on $C(U)$ satisfying
\begin{equation*}
(v_{1}\cdots v_{e})^{\dagger}=v_{e}\cdots v_{1}
\end{equation*}
for $v_1,\ldots,v_{e}\in U$. 

We define an algebraic group $\GSpin(U)$ over $k$ as
\begin{equation*}
\GSpin(U)(R)=\{g\in (C^{+}(U)\otimes_{k}R)^{\times}\mid gU_{R}g^{-1}=U_{R},g^{\dagger}g\in \G_m\}
\end{equation*}
for any $k$-algebra $R$. Moreover, put
\begin{equation*}
\sml_{U}\colon \GSpin(U)\rightarrow \G_m;g\mapsto g^{\dagger}g. 
\end{equation*}
Then we have an exact sequence
\begin{equation*}
1\rightarrow \G_m\rightarrow \GSpin(U)\xrightarrow{g\mapsto g\bullet} \SO(U)\rightarrow 1,
\end{equation*}
where $g\bullet v:=gvg^{-1}$. 

\begin{prop}\label{gspi}(\cite[\S 2]{Asgari2002})
\emph{
\begin{enumerate}
\item The $k$-group $\GSpin(U)$ is reductive and connected. 
\item The derived group of $\GSpin(U)$ is the kernel of $\sml_{U}$, which is simply connected. 
\item Suppose that there is an isomorphism of quadratic spaces
\begin{equation*}
U\cong
\begin{cases}
\H_{k}^{\oplus \dim_{k}(U)/2}&\text{if }2\mid n,\\
\H_{k}^{\oplus (\dim_{k}(U)-1)/2}\oplus l_{1}&\text{if }2\nmid n. 
\end{cases}
\end{equation*}
Then $\GSpin(U)$ is split as $k$-group. 
\end{enumerate}}
\end{prop}

We give a closed immersion of $G^{\sharp}$ into a symplectic similitude group. 

\begin{prop}\label{spsp}(\cite[4.1.4]{Howard2017})
\emph{Let $\delta_0 \in C(U)$ satisfying $\delta_0^{\dagger}=-\delta_0$, and set
\begin{equation*}
\psi_{\delta_0}\colon C(U)\times C(U)\rightarrow \Qp;(x,x')\mapsto \trd_{C(U)/\Qp}(x\delta_0 {x'}^{\dagger}). 
\end{equation*}
Then the following are true.  
\begin{enumerate}
\item The bilinear map $\psi_{\delta}$ is a non-degenerate symplectic form of $C(U)$. 
\item The homomorphism
\begin{equation*}
i_{\delta}\colon \GSpin(U)\rightarrow \GSp(C(U),\psi_{\delta});g\mapsto [x\mapsto gx]. 
\end{equation*}
is a closed immersion, and $\sml_{U}$ equals the composite of $i_{\delta}$ and the similitude character of $\GSp(C(U),\psi_{\delta_{0}})$. 
\end{enumerate}}
\end{prop}

Next, we consider the case $k=\Qp$ where $p>2$. 

\begin{prop}\label{witt}
\emph{There is $d\in \Znn$, an anisotropic subspace $U_{\an}$ of $U$ and an isomorphism
\begin{equation*}
U\cong \H_{\Qp}^{\oplus d}\oplus U_{\an},
\end{equation*}
where $U_{\an}$ is an anisotropic subspace of $U$. Moreover, we have $\dim_{\Qp}U_{\an}\leq 4$. }
\end{prop}

\begin{proof}
The assertion on the decomposition of $U$ is \cite[Lemma 22.3]{Shimura2010}. The inequality $\dim_{\Qp}U_{\an}\leq 4$ is a consequence of \cite[\S 25]{Shimura2010}. 
\end{proof}

\begin{cor}\label{sptm}
\emph{The $\Qp$-group $\GSpin(U)$ splits over a tamely ramified extension. }
\end{cor}

\begin{proof}
By Propositions \ref{witt}, the $\Qps(\sqrt{p})$-quadratic space $U\otimes_{\Qp}\Qps(\sqrt{p})$ satisfies the hypothesis of Proposition \ref{gspi} (ii). The field extension $\Qps(\sqrt{p})/\Qp$ is tamely ramified since $p>2$. 
\end{proof}

We give an explicit description of $\pi_{1}(G)$ and the Kottwitz map $\widetilde{\kappa}_{G}\colon G(\Qpb)\rightarrow \pi_{1}(G)$ defined in \cite[\S 7]{Kottwitz1997}. 

\begin{prop}\label{kott}
\emph{
\begin{enumerate}
\item There is an isomorphism $\pi_1(G)\cong \Z$, where the action of $\Gamma$ on $\Z$ is trivial. 
\item There is a commutative diagram
\begin{equation*}
\xymatrix{
G(\Qpb)\ar[d]\ar[r]^{\widetilde{\kappa}_{G}}&\pi_1(G)_{I}\ar[d]\\
\G_m(\Qpb)\ar[r]^{\quad \ord_{p}}& \Z. }
\end{equation*}
\end{enumerate}}
\end{prop}

\begin{proof}
(i): This follows from Proposition \ref{gspi} (ii). 

(ii): This is a construction of $\widetilde{\kappa}$ as in \cite[7.4]{Kottwitz1997}. 
\end{proof}

We give two invariants of quadratic spaces. Let
\begin{equation*}
(\,,\,)_p\colon \Qpt \times \Qpt \rightarrow \{\pm 1\}
\end{equation*}
be the Hilbert symbol of $\Qp$. Since $p>2$, there is $a'_1,\ldots,a'_n\in \Qpt$ and a basis $v_1,\ldots,v_n$ of $U$ whose Gram matrix of $[\,,\,]$ with respect to $v_1,\ldots,v_n$ is 
\begin{equation*}
\begin{pmatrix}
a'_1&&\\
&\ddots&\\
&&a'_n
\end{pmatrix}. 
\end{equation*}
Then we define the discriminant $\disc(U)\in \Qpt/(\Qpt)^{2}$ and the Hasse invariant $\varepsilon(U)\in \{\pm 1\}$ as follow: 
\begin{equation*}
\disc(V):=2^na'_1\cdots a'_n\bmod (\Qpt)^2,\quad \varepsilon(U):=\prod_{i<j}(a'_i,a'_j)_{p}. 
\end{equation*}
It is known that both $\disc(U)$ and $\varepsilon(U)$ are independent of the choice of the basis as above. Moreover, Proposition \ref{witt} and \cite[\S 25]{Shimura2010} imply that quadratic spaces of dimension $n$ is uniquely determined by its discriminant and Hasse invariant. 

Let $(V,Q)$ be a quadratic space of dimension $n$ over $\Qp$. We assume the following: 
\begin{itemize}
\item $V$ admits an almost self-dual lattice, that is, there is a $\Zp$-lattice $\bLambda$ of $V$ satisfying $\bLambda \subset \bLambda^{\vee}\subset p^{-1}\bLambda$ and $\length_{\Zp}(\bLambda^{\vee}/\bLambda)=1$. 
\end{itemize}

\begin{prop}\label{alsd}
\emph{There is a $\Qp$-basis $x_1,\ldots,x_n$ of $V$ and $a_1,\ldots,a_n\in \Zpt$ satisfying the following: 
\begin{enumerate}
\item $\bLambda=\Zp x_1\oplus \cdots \oplus \Zp x_n$,
\item the Gram matrix of $[\,,\,]$ with respect to $x_1,\ldots,x_n$ equals 
\begin{equation*}
\begin{pmatrix}
pa_1&&&&\\
&a_2&&&\\
&&a_3&&\\
&&&\ddots&\\
&&&&a_n
\end{pmatrix}. 
\end{equation*}
\end{enumerate}}
\end{prop}

We relate $V$ with another quadratic space. Define a quadratic space $(V^{\sharp},Q^{\sharp})$ over $\Qp$ by
\begin{equation*}
V^{\sharp}:=l_{-pa_1/2}\oplus V. 
\end{equation*}
Put $x_0:=(1,0)\in V^{\sharp}$ and
\begin{equation*}
x_i^{\sharp}:=
\begin{cases}
2^{-1}(x_1+x_0)&\text{if }i=0,\\
(pa_1)^{-1}(x_1-x_0)&\text{if }i=1,\\
x_{i}&\text{otherwise. }
\end{cases}
\end{equation*}

\begin{prop}\label{spsd}
\emph{
\begin{enumerate}
\item The Gram matrix of $[\,,\,]^{\sharp}$ with respect to $x_0^{\sharp},\ldots,x_{n}^{\sharp}$ is
\begin{equation*}
\begin{pmatrix}
0&1&&&&\\
1&0&&&&\\
&&a_2&&&\\
&&&a_3&&\\
&&&&\ddots&\\
&&&&&a_n
\end{pmatrix}. 
\end{equation*}
\item Let $\bLambda^{\sharp}:=\Zp x_0^{\sharp}\oplus \cdots \oplus \Zp x_{n}^{\sharp}$. Then $\bLambda^{\sharp}$ is a self-dual $\Zp$-lattice in $V^{\sharp}$ satisfying $\bLambda=\bLambda^{\sharp} \cap V$. 
\item If $n$ is odd, then we have $\disc(V^{\sharp})=(-1)^{(n+1)/2}$ if and only $\varepsilon(V)=(p,-1)_{p}^{(n-1)/2}$. 
\end{enumerate}}
\end{prop}

\begin{proof}
The assertions (i) and (ii) are clear. On the other hand, (iii) follows from the equalities $\disc(V^{\sharp})=(-1)\cdot (-a_2\cdots a_n)$ and $\varepsilon(V)=(p,-a_2\cdots a_n)_{p}$. 
\end{proof}

Put $G:=\GSpin(V)$, and denote by $\bx \in \cB(G,\Qp)$ the vertex corresponding to $\bLambda$. On the other hand, let $G^{\sharp}:=\GSpin(V^{\sharp})$, and write $\bx^{\sharp}\in \cB(G^{\sharp},\Qp)$ for the vertex corresponding to $\bx^{\sharp}$. Then we have $\cG_{\bx}=\cG_{\bx}^{\circ}$ and $\cG_{\bx^{\sharp}}=\cG_{\bx^{\sharp}}^{\circ}$.  

\begin{prop}\label{hseb}
\emph{
\begin{enumerate}
\item The inclusion $V\subset V^{\sharp}$ induces a closed immersion $G \hookrightarrow G^{\sharp}$. 
\item The morphism $\cB(G,\Qp)\rightarrow \cB(G^{\sharp},\Qp)$ induced by $G \hookrightarrow G^{\sharp}$ maps $\bx$ to $\bx^{\sharp}$. 
\item The closed immersion $G \hookrightarrow G^{\sharp}$ induces a closed immersion $\cG_{\bx}\hookrightarrow \cG_{\bx^{\sharp}}$. An element $g\in \cG_{\bx^{\sharp}}$ lies in the image of $\cG_{\bx}\hookrightarrow \cG_{\bx^{\sharp}}$ if and only if $gx_0g^{-1}=x_0$. 
\end{enumerate}}
\end{prop}

\begin{proof}
(i): This follows from the definition of spinor similitude groups. 

(ii): This is a consequence of Proposition \ref{spsd} (ii). 

(iii): The assertions follow from \cite[Lemma 2.6]{MadapusiPera2016}. 
\end{proof}

We denote by $C$ the finite free $\Zp$-module $C(\bLambda^{\sharp})$, and put $C_0:=C\otimes_{\Zp}\Qp=C(V^{\sharp})$. Fix $\delta_0 \in C$ and consider $\psi^{\sharp}=\psi_{\delta}$ and $i^{\sharp}=i_{\delta_0}$ as in Proposition \ref{spsp}. Then the $\Zp$-lattice $C\subset C_0$ is self-dual with respect to $\psi^{\sharp}$. Put $G':=\GSp(C_0)$, and denote by $\bx'\in \cB(G',\Qp)$ the vertex corresponding to the self-dual lattice $C$. Then $i^{\sharp}$ maps $\bx^{\sharp}$ to $\bx'$. Hence we obtain a closed immersion $\cG_{\bx^{\sharp}}\hookrightarrow \cG_{\bx'}$. On the other hand, we denote by $i$ the composite
\begin{equation*}
G\hookrightarrow G^{\sharp}\xrightarrow{i^{\sharp}}G'. 
\end{equation*}
In the sequel of this paper, we regard $C$ as a representation of $\cG_{\bx^{\sharp}}$ by $i^{\sharp}$, or as that of $\cG_{\bx}$ by $i$. Let $C^{*}:=\Hom_{\Zp}(C,\Zp)$ be the contragradient representation of $C$. We regard $C$ as a $\Zp$-subspace of the $\cG_{\bx^{\sharp}}$-representation $\End_{\Zp}(C^{*})\cong C\otimes_{\Zp}C^{*}$ by the map
\begin{equation*}
\bLambda^{\sharp}\hookrightarrow \End_{\Qp}(C^{*})
\end{equation*}
defined as $(c_1\cdot f)(x)=f(c_1x)$. Then the action of $g\in \cG_{\bx^{\sharp}}$ on $v\in \bLambda^{\sharp}$ can be written as $g\bullet v$. 

\begin{dfn}\label{tnsr}
Let $(s^{\sharp}_{\alpha^{\sharp}})$ be a finite collection $(s^{\sharp}_{\alpha^{\sharp}})$ of elements of $C^{\otimes}$ whose stabilizer in $\cG_{\bx_0}$ equals $\cG_{\bx^{\sharp}}$. Then we write $(s_{\alpha})$ for the (disjoint) union of $(s^{\sharp}_{\alpha^{\sharp}})$ and $x_0 \in C\subset \End_{\Zp}(C^{*})$. 
\end{dfn}

In the sequel of this paper, we fix $(s^{\sharp}_{\alpha^{\sharp}})$ as in Definition \ref{tnsr}, which determines $(s_{\alpha})$. By Proposition \ref{hseb} (iii), $\cG_{\bx}$ is the stabilizer of $(s_{\alpha})$. 

\subsection{Local models}

Until the end of Section \ref{spfb}, we use the notation $\star \in \{\emptyset,\sharp\}$. Let $V^{\star}$, $G^{\star}$ and $G'$ be as in Section \ref{quad}. %, and put $H^{\star}=\SO(V^{\star})$. 

First, we specify minuscule cocharacters of $G^{\star}$. Let
\begin{equation*}
\mu^{\sharp}\colon \G_m\rightarrow G^{\sharp}; t \mapsto t^{-1}x_1^{\sharp}x_2^{\sharp}+x_2^{\sharp}x_1^{\sharp}. 
\end{equation*}

\begin{lem}\label{hsmn}(\cite[4.2.1]{Howard2017})
\emph{The cocharacter $\mu^{\sharp}$ of $G^{\sharp}$ is minuscule, and is defined over $\Qp$. Moreover, we have the following: 
\begin{equation*}
(i^{\sharp}\circ \mu^{\sharp})(t)(x)=
\begin{cases}
t^{-1}x&\text{if }x\in x_1^{\sharp}C_0^{\sharp},\\
x&\text{if }x\in x_2^{\sharp}C_0^{\sharp},
\end{cases}
\quad \mu^{\sharp}(t)\bullet x_i^{\sharp}=
\begin{cases}
t^{(-1)^i}x_i^{\sharp}&\text{if }i=1,2,\\
x_i^{\sharp}&\text{if }3\leq i\leq n+1.
\end{cases}
\end{equation*}}
\end{lem}

On the other hand, fix a square root $\alpha$ of $-a_2a_3^{-1}$ in $\Qps$, and set
\begin{equation*}
\mu(t):=t^{-1}(x_2+\alpha x_3)(x_2-\alpha x_3)+(x_2-\alpha x_3)(x_2+\alpha x_3). 
\end{equation*}

\begin{lem}\label{mucp}
\emph{
\begin{enumerate}
\item The cocharacter $\mu$ of $G$ is minuscule, and is defined over $\Qps$. 
\item The composite
\begin{equation*}
\G_m \xrightarrow{\mu}G\hookrightarrow G^{\sharp}
\end{equation*}
is conjugate to $\mu^{\sharp}$ by construction by an element of $\cG_{\bx^{\sharp}}(\Zps)$. 
\end{enumerate}}
\end{lem}

\begin{proof}
(i): This follows by definition. 

(ii): The map
\begin{equation*}
g_0\colon \bLambda^{\sharp}\otimes_{\Zp}\Zps \rightarrow \bLambda^{\sharp}\otimes_{\Zp}\Zps ;x_i^{\sharp}\mapsto 
\begin{cases}
x_3^{\sharp}+(-1)^{i-1}\alpha x_4^{\sharp}&\text{if }i=1,2,\\
(2\alpha^{i-3})^{-1}(x_{1}^{\sharp}+(-1)^{i-1}x_2^{\sharp})&\text{if }i=3,4,\\
x_i^{\sharp}&\text{otherwise}
\end{cases}
\end{equation*}
is an element of $\cG_{\bx^{\sharp}}(\Zps)$ which satisfies $\Ad(g_0)\circ \mu=\mu^{\sharp}$. 
\end{proof}

Let $M^{\loc}(G,\sigma(\mu^{\star}),\bx^{\star})$ be the local model for $(G^{\star},\mu^{\star},\bx^{\star})$ in the sense of \cite{Pappas2013}. We denote by $M^{\loc}(G')$ the moduli space of maximally totally isotropic subspaces of $C$. 

\begin{prop}\label{cllm}(\cite[Proposition 2.3.7]{Kisin2018})
\emph{The closed immersion $i^{\star}\colon G^{\star}\hookrightarrow G'$ induces a closed immersion
$i_{*}^{\star}\colon M^{\loc}(G^{\star},\mu^{\star},\bx^{\star})\hookrightarrow M^{\loc}(G')$. }
\end{prop}

\begin{cor}\label{lmci}
\emph{The inclusion $\bLambda \subset \bLambda^{\sharp}$ induces a closed immersion
\begin{equation*}
M^{\loc}(G,\sigma(\mu),\bx)\hookrightarrow M^{\loc}(G^{\sharp},\mu^{\sharp},\bx^{\sharp})
\end{equation*}
such that the following diagram commutes: 
\begin{equation*}
\xymatrix{
M^{\loc}(G,\sigma(\mu),\bx)\ar[rd]_{i_{*}} \ar[r] & M^{\loc}(G^{\sharp},\mu^{\sharp},\bx^{\sharp}) \ar[d]^{i_{*}^{\sharp}} \\
&M^{\loc}(G'). }
\end{equation*}}
\end{cor}

\begin{proof}
By a descriptions of $M^{\loc}(G,\sigma(\mu),\bx)$ and $M^{\loc}(G^{\sharp},\mu^{\sharp},\bx^{\sharp})$ as in \cite[4.1.5]{Kisin2018}, $\bLambda \subset \bLambda^{\sharp}$ induces a morphism $M^{\loc}(G,\sigma(\mu),\bx)\hookrightarrow M^{\loc}(G^{\sharp},\mu^{\sharp},\bx^{\sharp})$. This is a closed immersion by Proposition \ref{cllm}. 
\end{proof}

\begin{cor}\label{admr}
\emph{We have $(\Kb^{\sharp}_p\mu^{\sharp}(p)\Kb^{\sharp}_p)\cap G(\Qpb)=\bigcup_{w\in \Adm(\mu)}\Kb_p \sigma(w) \Kb_p$. }
\end{cor}

\begin{proof}
This follows from Proposition \ref{lmci} and the isomorphisms
\begin{equation*}
M^{\loc}(G^{\sharp},\mu^{\sharp},\bx^{\sharp})(\Fpbar)\cong \Kb_{p}^{\sharp}\mu^{\sharp}(p)\Kb_{p}^{\sharp}/\Kb_p^{\sharp},\quad  M^{\loc}(G,\sigma(\mu),\bx)(\Fpbar)\cong \bigcup_{w\in \Adm(\mu)}\Kb_{p}\sigma(w)\Kb_{p}/\Kb_p,
\end{equation*}
that are consequences of \cite[Proposition 3.4]{Zhou2020}. 
\end{proof}

We give a more precise description of $M^{\loc}(G^{\star},\mu^{\star},\bx^{\star})_{\Zpb}:=M^{\loc}(G^{\star},\mu^{\star},\bx^{\star})\otimes_{\Zp}\Zpb$ and $i_{*}^{\star}$. 

\begin{prop}\label{lmex}
\emph{
\begin{enumerate}
\item There is an isomorphism between $M^{\loc}(G^{\sharp},\mu^{\sharp},\bx^{\sharp})_{\Zpb}$ and the moduli space of isotropic lines in $\bLambda^{\sharp}\otimes_{\Zp}\Zpb$, which is smooth over $\Zpb$. 
\item Under (i), the closed immersion $i_{*}^{\sharp}\colon M^{\loc}(G^{\sharp},\mu^{\sharp},\bx^{\sharp})_{\Zpb}\hookrightarrow M^{\loc}(G')_{\Zpb}$ equals the map $l \mapsto \cF(l)$, where
\begin{equation*}
\cF(l):=\{x\in C\mid v\cdot x=0\text{ for all }v\in l\}. 
\end{equation*}
\end{enumerate}}
\end{prop}

\begin{proof}
By \cite[Proposition 1.10]{MadapusiPera2016}, the parabolic subgroup of $H^{\sharp}$ associated to $t \mapsto \mu^{\sharp}(t)\bullet$ corresponds to that of $G^{\sharp}$ associated to $\mu^{\sharp}$, which is the stabilizer of $x_1^{\sharp}C_0^{\sharp}$. Hence the assertions follow by \cite[4.1.5]{Kisin2018}. 
\end{proof}

\begin{dfn}\label{cplr}
We define a complete local $\Zpb$-algebra as follows: 
\begin{equation*}
R_n:=\Zpb[[t_1,\ldots,t_{n-1}]]/(p+\sum_{i=1}^{n-1}t_it_{n-i}). 
\end{equation*}
\end{dfn}

\begin{thm}\label{qtrc}
\emph{
\begin{enumerate}
\item There is an isomorphism between $M^{\loc}(G,\sigma(\mu),\bx)_{\Zpb}$ and the moduli space of isotropic lines in $\bLambda$. It is regular of dimension $n-1$, is flat over $\Zp$. 
\item The scheme $M^{\loc}(G,\sigma(\mu),\bx)_{\Zpb}$ is smooth over $\Zpb$ outside the unique $\Fpbar$-valued point $y_0$ corresponding to the radical of $\bLambda \otimes_{\Zp}\Fpbar$. Moreover, there is an isomorphism 
\begin{equation*}
\widehat{M}^{\loc}(G,\sigma(\mu),\bx)_{\Zpb,y_0}\cong \spf R_n. 
\end{equation*}
\item The closed immersion $M^{\loc}(G,\sigma(\mu),\bx)\hookrightarrow M^{\loc}(G^{\sharp},\mu^{\sharp},\bx^{\sharp})$ is the canonical morphism, that is, the map defined by regarding an isotropic line $l$ in $\bLambda$ as that in $\bLambda^{\sharp}$. 
\end{enumerate}}
\end{thm}

\begin{proof}
These follow from \cite[Proposition 12.6, Proposition 12.7]{He2020a} and their proofs. 
\end{proof}

\begin{cor}\label{sghd}
\emph{Let $y\in M^{\loc}(G,\sigma(\mu),\bx)(\Fpbar)$. Then $y$ is a non-smooth point of $M^{\loc}(G,\sigma(\mu),\bx)$ if and only if $x_0(C\otimes_{\Zp}\Fpbar)=\cF_{i_{*}(y)}$, where $\cF_{i_{*}(y)}$ is the maximally totally isotropic subspace of $C\otimes_{\Zp}\Fpbar$ corresponding to $i_{*}(y)\in M^{\loc}(G')(\Fpbar)$. }
\end{cor}

\begin{proof}
By Proposition \ref{lmex} and Theorem \ref{qtrc}, $y$ is a non-smooth point of $M^{\loc}(G,\sigma(\mu),\bx)$ if and only if $x_0(C\otimes_{\Zp}\Fpbar)\subset \cF(l_y)$. On the other hand, it is equivalent to the equality $x_0(C\otimes_{\Zp}\Fpbar)=\cF(l_y)$ since both $x_0(C\otimes_{\Zp}\Fpbar)$ and $\cF(l_y)$ have the same dimension $2^n$. 
\end{proof}

\subsection{$p$-divisible groups and their deformations}\label{dfpd}

Here we recall the theory of deformations of $p$-divisible groups with crystalline tensors, which is depeloped in \cite[\S 3]{Kisin2018}, in the case for spinor groups. 

For $b,b'\in G^{\star}(\Qpb)$, we say that $b$ and $b'$ are \emph{$\sigma$-conjugate} if there is $g\in G^{\star}(\Qpb)$ satisfying $b'=gb\sigma(g)^{-1}$. We write $B(G)$ for the set of $\sigma$-conjugacy classes of $G(\Qpb)$. 

We denote by $\T$ the pro-torus over $\Qpb$ whose cocharacter group is $\Q$, and let
\begin{equation*}
\overline{\nu}_{G^{\star}}([b])\colon B(G^{\star})\rightarrow \Hom_{\Qpb}(\T,G\otimes_{\Qp}\Qpb)
\end{equation*}
be the map constructed in \cite[\S 4]{Kottwitz1985}, which is called the \emph{Newton map}. As in \cite[5.1]{Kottwitz1985}, we say that $[b]$ is \emph{basic} if $\overline{\nu}_{G^{\star}}([b])$ factors through the center $\G_m$ of $G$. Then $\overline{\nu}_{G^{\star}}([b])$ corresponds to an element of $\Q$ since
\begin{equation*}
\Hom_{\Qp}(\T,\G_m)\cong \Q. 
\end{equation*}
Moreover, the proof of \cite[Proposition 4.2.5]{Howard2017} implies that $\overline{\nu}_{G^{\star}}([b])$ can be written as 
\begin{equation*}
B(G^{\star})\rightarrow \frac{1}{2}\Z; [b]\mapsto \frac{\ord_{p}(\sml_{V^{\star}}(b))}{2},
\end{equation*}
where $\ord_{p}\colon \breve{\mathbb{Q}}_p^{\times}\rightarrow \Z$ is the $p$-adic valuation satisfying $\ord_{p}(p)=1$. 

\begin{prop}\label{bseq}
\emph{For $b\in G^{\star}(\Qpb)$, the following are equivalent: 
\begin{enumerate}
\item $[b]\in B(G^{\star})$ is basic,
\item the isocrystal $(C_0^{*}\otimes_{\Qp}\Qpb,b\sigma^{*})$ is isoclinic,
\item the isocrystal $(V^{\star}\otimes_{\Qp}\Qpb,\Ad(b\sigma))$ is isoclinic of slope $0$. 
\end{enumerate}
Moreover, if the above equivalent conditions hold, then the slope of $(C^{*},b\sigma^{*})$ is $-\overline{\nu}([b])$. }
\end{prop}

\begin{proof}
The proof is the same as \cite[Lemma 4.2.4]{Howard2017}. 
\end{proof}

We give a basic element of $B(G^{\star})$ explicitly. Take a square root $\alpha_0$ of $-a_1(2a_2)^{-1}$, and put
\begin{equation*}
b_0:=\frac{2\alpha_0}{pa_1}x_2x_1\in G(\Qps). 
\end{equation*}

\begin{prop}\label{hsbs}
\emph{
\begin{enumerate}
\item We have $b_0=\alpha_0 x_2^{\sharp}(2(pa_1)^{-1}x_0^{\sharp}+x_1^{\sharp})\in G^{\sharp}(\Qps)$. 
\item Let $\Phi^{\sharp}:=\Ad(b_0\sigma)$ be the $\sigma$-linear map on $\L_0^{\sharp}$. Then we have
\begin{equation*}
\Phi^{\sharp}(x_i^{\sharp})=
\begin{cases}
-(pa_1/2)^{(-1)^{i}}x_{1-i}^{\sharp}&\text{if }i=0,1,\\
-x_2^{\sharp}&\text{if }i=2,\\
x_i^{\sharp}&\text{otherwise}. 
\end{cases}
\end{equation*}
\item The element $b_0$ is decent which is contained in $\Kb_p^{\sharp}\mu^{\sharp}(p)\Kb_p^{\sharp}$ which is decent. 
\item The $\sigma$-conjugacy class $[b_0]\in B(G^{\sharp})$ is basic. Moreover, we have $\overline{\nu}_{G^{\sharp}}([b_0])=-1/2$. 
\end{enumerate}}
\end{prop}

\begin{proof}
(i), (ii): These follows by direct computation. 

(iii): The proof of $b_0 \in \Kb_p^{\sharp}\mu^{\sharp}(p)\Kb_p^{\sharp}$ is the same as that of \cite[Proposition 4.2.6]{Howard2017}. The decentness is a consequence of the equality $b_0^2=\pm p^{-1}$, which can be checked by calculation. 

(iv): This follows from (iii). 
\end{proof}

\begin{prop}\label{spbs}
\emph{
\begin{enumerate}
\item Let $\Phi:=\Ad(b_0\sigma)$ be the $\sigma$-linear map on $V\otimes_{\Qp}\Qpb$. Then we have
\begin{equation*}
\Phi(x_i)=
\begin{cases}
-x_i&\text{if }i=1,2,\\
x_i&\text{otherwise}. 
\end{cases}
\end{equation*}
\item The element $b_0$ is decent, and is contained in $\bigcup_{w\in \Adm(\mu)}\Kb_p\sigma(w)\Kb_p$. 
\item The $\sigma$-conjugacy class $[b_0]\in B(G)$ is basic. Moreover, we have $\overline{\nu}_{G}([b_0])=-1/2$. 
\end{enumerate}}
\end{prop}

\begin{proof}
(i): This follows by definition. 

(ii): By Proposition \ref{hsbs} (iii), we obtain that $b_0$ is decent and $b_0\in \Kb_p^{\sharp}\mu^{\sharp}(p)\Kb_p^{\sharp}$. This implies $b_0\in \bigcup_{w\in \Adm(\mu)}\Kb_p\sigma(w)\Kb_p$ by Corollary \ref{admr}. 

(iii): This follows from Propositions \ref{bseq} and \ref{hsbs} (iv). 
\end{proof}

\begin{rem}
Let $B(G^{\star},\{\mu^{\star}\})$ be the finite set of $B(G^{\star})$ defined as in \cite[6.2]{Kottwitz1997}. Then \cite[Theorem 2.1]{He2016} implies $[b_0]\in B(G^{\star},\{\sigma(\mu^{\star})\})=B(G^{\star},\{\mu^{\star}\})$. Hence $[b_0]$ is the unique basic element of $B(G^{\star})$ which is containied in $B(G^{\star},\{\mu^{\star}\})$. 
\end{rem}

We recall an existence of certain $p$-divisible groups over $\Fpbar$ which will be used in later. For a $p$-divisible group $X$ over $\Fpbar$, we denote by $\D(X)$ the \emph{contavariant} Dieudonn{\'e} module of $X$. We write $F$ and $V$ for its Frobenius and Vershibung respectively. Then $\D(X)_{\Fpbar}:=\D(X)\otimes_{\Zpb}\Fpbar$ is equipped with the Hodge filtration $\Fil^1(X)=(\Lie(X))^{*}\subset \D(X)_{\Fpbar}$, which induces an isomorphism $\D(X)_{\Fpbar}/\Fil^1(X)\cong \Lie(X^{\vee})$. Moreover, there are isomorphisms
\begin{equation*}
\Fil^{1}(X)\cong V(\D(X))/p\D(X),\quad \Lie(X^{\vee})\cong \D(X)/V(\D(X)). 
\end{equation*}

In the sequel, we regard that the dual $\D(X)^{*}$ of $\D(X)$ is equpped with the action of $F$ by the contragradient. Note that $\D(X)^{*}$ is no longer stable under $F$. 

\begin{prop}\label{expd}
\emph{There is a triple $(\X_0,\lambda_0,\tau_0)$, where
\begin{itemize}
\item $\X_0$ is a $p$-divisible group over $\Fpbar$,
\item $\lambda_0 \colon \X_0\rightarrow \X_0^{\vee}$ is a principal polarization,
\item $\tau_0\colon \D(\X_0)\xrightarrow{\cong} C^{*}\otimes_{\Zp}\Zpb$ is an isomorphism,
\end{itemize}
which satisfies the following: 
\begin{itemize}
\item $F$ corresponds to $b\sigma^{*}$ under $\tau_0$, 
\item $\tau_0(\Fil^1(X))$ is induced by a conjugate of the cocharacter $\mu^{\sharp}\otimes \id_{\Fpbar}\colon \G_m\rightarrow \cG_{\bx}\otimes_{\Zp}\Fpbar$. 
\end{itemize}
Moreover, such a triple $(\X_0,\lambda_0,\tau_0)$ is unique up to isomorphism. }
\end{prop}

\begin{proof}
The proof is the same as that of \cite[Lemma 2.2.5]{Howard2017}. 
\end{proof}

Let $(X,\lambda)$ be a polarized $p$-divisible group over $\Fpbar$. We assume the following: 

\begin{axm}\label{exct}
There is a collection of $F$-invariant tensors $(t_{\alpha^{\star},0}^{\star})$ in $\D(X)^{\otimes}$ and an isomorphism $\tau \colon \D(X)\xrightarrow{\cong} C^{*}\otimes_{\Zp}\Zpb$ satisfying the following: 
\begin{enumerate}
\item under the isomorphism $\tau$, $\lambda$ and $t_{\alpha^{\star}}^{\star}$ correspond to a $\Qpbt$-multiple of $\psi^{\sharp}$ and $s_{\alpha^{\star}}^{\star}\otimes 1$ respectively,
\item the point of $M^{\loc}(G')$ corresponding to $\tau(\Fil^{1}(X))$ lies in the image of $M^{\loc}(G,\mu,\bx)$ under $i_{*}^{\sharp}$. 
\end{enumerate}
\end{axm}

On the other hand, Lemmas \ref{hsmn}, \ref{mucp} (ii), Corollary \ref{sptm} and Proposition \ref{gspi} (ii) imply the all hypothesis appeared in \cite[\S 3.3]{Hamacher2019} are satisfied. Hence we can apply the theory in \cite[\S 3]{Kisin}. 

We denote by $y$ the point of $M^{\loc}(G^{\star},\mu^{\star},\bx^{\star})$ in Axiom \ref{exct} (ii). 

\begin{dfn}
\begin{enumerate}
\item Let $\Def(X,\lambda)$ be the deformation of the polarized $p$-divisible group $(X,\lambda)$. Then the Grothendieck--Messing theory implies that there is an isomorphism $\Def(X,\lambda)\cong \widehat{M}^{\loc}(G')_{i_{*}^{\star}(y)}$. 
\item We define $\Def(X,\lambda,(t_{\alpha^{\star}}^{\star}))$ as the unique closed formal subscheme of $\Def(X,\lambda)$ corresponding to $\widehat{M}^{\loc}(G^{\star},\mu^{\star},\bx^{\star})_{y}\xrightarrow{i_{*}^{\star}} \widehat{M}^{\loc}(G')_{i_{*}^{\star}(y)}$. We call $\Def(X,\lambda,(t_{\alpha^{\star}}^{\star}))$ as the \emph{deformation of $(X,\lambda,(t_{\alpha^{\star}}^{\star}))$}. 
\end{enumerate}
\end{dfn}

Finally, we consider the case $\star=\emptyset$. We write $\iota_{X}(x_0)$ for the isogeny on $X$ corresponding to the tensor $\tau^{-1}(x_0\otimes 1)$. 

\begin{prop}\label{sgl0}
\emph{Suppose that $y\in M^{\loc}(G,\mu,\bx)$ is the unique non-smooth point, that is, $\iota_{X}(x_0)$ acts as the zero map on $\Lie(X^{\vee})$. Then $X$ is isoclinic of slope $1/2$. In particular, if we write $F=b\sigma^{*}$ where $b\in G(\Qpb)$, then $[b]=[b_0]$. }
\end{prop}

\begin{proof}
By assumption, we have $V(\D(X_0))=\iota_{X}(x_0)(\D(X_0))$. This implies $F^2(\D(X))=p\D(X)$, which means that $X$ is isoclinic of slope $1/2$. Moreover, Proposition \ref{bseq} asserts that $[b]\in B(G)$ is basic.  
\end{proof}

\section{Kisin--Pappas integral models}\label{kpim}

\subsection{Definition of the Kisin--Pappas integral models}\label{dfkp}

Let $\bV$ be a quadratic space over $\Q$ of signature $(d,2)$ where $d\in \Zpn$. We assume that $V:=\bV\otimes_{\Q}\Q$ satisfies (A) as in Section \ref{quad}. Put $\bG:=\GSpin(\bV)$, the spinor similitude group of $\bV$, and let $K_p\subset \bG(\Qp)$ be as in Section \ref{quad}. 

\begin{lem}\label{sphs}
\emph{There is $a\in \Q_{>0}$ such that $\bV^{\sharp}:=\bV\oplus l_{a/2}$ satisfies $\bV^{\sharp}\otimes_{\Q}\Qp\cong V^{\sharp}$ as quadratic spaces over $\Qp$. }
\end{lem}

\begin{proof}
Take a basis $x_1,\ldots,x_n$ of $V$ as Proposition \ref{alsd}. By replacing $x_1$ to its $\Zpt$-multiple, we may further assume $a_1\in \Q_{<0}$. Then $a:=-pa_1$ satisfies the desired condition. 
\end{proof}

Let $\bS$ be the Deligne torus, that is, the Weil restriction of $\G_{m,\C}$ to $\R$. Take an $\R$-basis $v_{1}^{+},\ldots,v_{d}^{+},v_{1}^{-},v_{2}^{-}\in \bV\otimes_{\Q}\R$ whose Gram matrix of the bilinear form associated to $\bV\otimes_{\Q}\R$ is
\begin{equation*}
\begin{pmatrix}
1&&&&\\
&\ddots&&&\\
&&1&&\\
&&&-1&\\
&&&&-1
\end{pmatrix}. 
\end{equation*}
We define $\bX^{\star}$ as the $\bG^{\star}(\R)$-conjugacy class of the $\R$-homomorphism
\begin{equation*}
\bS \rightarrow \bG^{\star}\otimes_{\Q}\R;a+b\sqrt{-1}\mapsto a+bv_{1}^{-}v_2^{-}. 
\end{equation*}
Then $(\bG^{\star},\bX^{\star})$ is a Shimura datum in the sense of \cite[2.1.1]{Deligne1979}. Moreover, the inclusion $\bV\subset \bV^{\sharp}$ induces an embedding of Shimura datum $(\bG,\bX)\hookrightarrow (\bG^{\sharp},\bX^{\sharp})$. 

We prove that $(\bG^{\sharp},\bX^{\sharp})$ is of Hodge type. Put $\bC_0:=C(V^{\sharp})$, and take ${\delta}_0\in \bC_0$ and consider the symplectic form $\psi_{\delta_0}$ on $\bC_0$ as in Proposition \ref{spsp} (i). Put $\bG':=\GSp(\bC_0,\psi_{\delta_0})$, the symplectic similitude group. Let
\begin{equation*}
v_{0,1}:=\frac{1}{\sqrt{2}}(v_{1}^{+}+v_{1}^{-}), \quad v_{0,2}:=\frac{1}{\sqrt{2}}(v_{1}^{+}-v_{1}^{-})
\end{equation*}
and write $S_{\bC_0}^{\pm}$ for the $\bG'(\R)$-conjugacy class of the $\R$-homomorphism
\begin{equation*}
\bS \rightarrow \bG'\otimes_{\Q}\R;a+b\sqrt{-1}\mapsto [v_{0,1}x+v_{0,2}x'\mapsto a(v_{0,1}x+v_{0,2}x')+b(v_{0,2}x'-v_{0,1}x)], 
\end{equation*}
where $x,x'\in \bC_0$. Then $(\bG',S_{\bC_0}^{\pm})$ is also a Shimura datum. Moreover, we have the following: 
\begin{prop}\label{hssh}
\emph{The closed immersion $i_{\delta_0}\colon \bG^{\sharp}\hookrightarrow \bG'$ in Proposition \ref{spsp} (ii) induces an embedding of Shimura datum
\begin{equation*}
\bi^{\sharp} \colon (\bG^{\sharp},\bX^{\sharp})\hookrightarrow (\bG',S_{C(\bV^{\sharp})}^{\pm}). 
\end{equation*}}
\end{prop}

We denote by $\bi$ the composite $\bG\hookrightarrow \bG^{\sharp}\xrightarrow{\bi^{\sharp}}\bG'$. Let $G^{\star}:=\bG^{\star}\otimes_{\Q}\Qp$, and $\bx^{\star}\in \cB(G^{\star})$ the vertex as in Section \ref{quad}. Put $K_p^{\star}:=\cG_{\bx^{\star}}(\Zp)$. Moreover, let $K^{\star,p}$ be a neat compact open subgroup of $\bG^{\star}(\A_{f}^{p})$, and put $K^{\star}:=K_p^{\star}K^{\star,p}$. In the following, we give a definition of Kisin--Pappas integral model of the Shimura variety $\Sh_{K}(\bG^{\star},\bX^{\star})$. For a neat compact open subgroup ${K'}^p$ of $\bG'(\A_{f}^{p})$, put $K':={K'}^pK'_p$. Here we require that the morphism
\begin{equation*}
\Sh_{K^{\star}}(\bG^{\star},\bX^{\star})\rightarrow \Sh_{K'}(\bG',S_{\bC_0}^{\pm})
\end{equation*}
induced by $\bi^{\star}$ is a closed immersion. Let $\sS_{K'}$ be the $\Z_{(p)}$-scheme defined as the moduli space of the functor which parametrizes prime-to-$p$ isogeny classes of triples $(A,\lambda,\overline{\eta}^p)$ for any connected noetherian scheme $S$ over $\Z_{(p)}$, where
\begin{itemize}
\item $A$ is an abelian scheme over $S$ of dimension $m$,
\item $\lambda$ is a prime-to-$p$ polarization of $A$,
\item $\overline{\eta}^p$ is a $K^p$-level structure on $A$. 
\end{itemize}
Note that $\sS_{K'}$ is flat over $\Z_{(p)}$, and its generic fiber is $\Sh_{K'}(\bG',S_{\bC_0}^{\pm})$. We define $\sS_{K^{\star}}$ as the normalization of the scheme-theoretic closure of $\Sh_{K^{\star}}(\bG^{\star},\bX^{\star})$ in $\sS_{K'}$. We call $\sS_{K^{\star}}$ the Kisin--Pappas integral model of $\Sh_{K^{\star}}(\bG^{\star},\bX^{\star})$ associated to $\bi^{\star}$. 

\subsection{Basic properties on the special fibers}\label{spfb}

We keep the notations as in Section \ref{dfkp}. Here we only consider the case $\star=\emptyset$. Let $(\cA_{K'},\lambda_{K'})$ be the universal polarized abelian scheme over $\sS_{K'}$, and $(\cA_{K},\lambda_{K})$ its pull-back by the canonical morphism $\sS_{K}\rightarrow \sS_{K'}$. Take $z\in \sS_{K}(\Fpbar)$, and denote by $(\cA_{z}[p^{\infty}],\lambda_{z})$ the fiber of $(\cA_{K},\lambda_{K})$ at $z$. Then, as discussed in \cite[\S 3.4]{Hamacher2019}, after replacing $(s_{\alpha})$ if necessary, there is a pair $((t_{\alpha,0,z}),\tau_{z})$, where
\begin{itemize}
\item $(t_{\alpha,0,z})$ is a collection of $\D(\cA_{z})^{\otimes}$,
\item $\tau_{z}\colon \D(\cA_{z})\xrightarrow{\cong} C^{*}\otimes_{\Zp}\Zpb$ is a $\Zpb$-isomorphism,
\end{itemize}
that satisfy Axiom \ref{exct}. Hence we can consider the deformation $\Def(\cA_{z}[p^{\infty}],\lambda_{z},(t_{\alpha,0,z}))$ of the triple $(\cA_{z},\lambda_{z},(t_{\alpha,0,z}))$. Moreover, it describes the complete neighborhood of $z$. 

\begin{prop}\label{dfp1} (\cite[Proposition 4.2.2, Corollary 4.2.4]{Kisin2018})
\emph{For $z\in \sS_{K}(\Fpbar)$, there is an isomorphism $\widehat{\sS}_{K,z}\cong \Def(\cA_{z}[p^{\infty}],\lambda_z,(t_{\alpha,0,z}))$. }
\end{prop}

On the other hand, under the isomorphism $\tau_{z}$ as above, the Frobenius on $\D(\cA_{z})$ corresponds to $b_{z}\sigma^{*}$ where $b_z \in G(\Qpb)$. By construction, the $\sigma$-conjugacy class $[b_z]\in B(G)$ is uniquely determined. Moreover, we have $[b_z]\in B(G,\{\mu\})$ by \cite[\S 8]{Zhou2020}. 

\begin{dfn}
We define a locally closed subscheme of $\sSbar_{K}$ as follows: 
\begin{equation*}
\sSbar_{K}^{\bs}:=\{z\in \sS_{K}\mid [b_z]=[b_0]\}. 
\end{equation*}
We call $\sSbar_{K}^{\bs}$ as the \emph{basic locus} of $\sSbar_{K}$.  
\end{dfn}
It is known that $\sSbar_{K}^{\bs}$ is a non-empty closed subset of $\sSbar_{K}$ by \cite[Theorem 1.3.13 (2)]{Kisin}. 

Finally, we give an explicit description the non-smooth locus $\sS_{K,\Zpb}^{\nsm}$ of $\sS_{K,\Zpb}:=\sS_{K}\otimes_{\Z_{(p)}}\Zpb$. We denote by $\iota_{z}(x_0)$ the isogeny on $\cA_{z}$ associated to the tensor $\tau^{-1}_{z}(x_0)$. 

\begin{thm}\label{sgbs}
\emph{
\begin{enumerate}
\item The $\Zp$-scheme $\sS_{K,\Zpb}$ is regular of dimension $n-1$, and is flat over $\Zp$. 
\item The non-smooth locus $\sS_{K,\Zpb}^{\nsm}$ is the set of $z\in \sS_{K}(\Fpbar)$ such that $\iota_{z}(x_0)$ acts by the zero map on $\Lie(\cA_{z}[p^{\infty}])$. Moreover, if the above conditions hold, there is an isomorphism $\widehat{\sS}_{K,\Zpb,z}\cong \spf R_n$. 
\item The non-smooth locus $\sS_{K,\Zpb}^{\nsm}$ is contained in $\sSbar_{K}^{\bs}$. 
\end{enumerate}}
\end{thm}

\begin{proof}
(i), (ii): These follow from Theorem \ref{qtrc} (i), (ii), Corollary \ref{sghd} and \cite[Theorem 4.2.7]{Kisin2018}. 

(iii): This is a consequence of Propositions \ref{dfp1} and \ref{sgl0}. 
\end{proof}

\section{Rapoport--Zink spaces}\label{rzsp}

In this section, we use the notations $\star \in \{\emptyset,\sharp,\,'\}$. We define the Rapoport--Zink spaces for $G^{\star}$ with level $K_p^{\star}$, which is based on \cite[\S 5]{Hamacher2019}. See also \cite[\S 4]{Oki2020b}. 

\subsection{Definition of the Rapoport--Zink spaces}\label{dfrz}

Here we use the notation $\star \in \{\emptyset,\sharp,\,'\}$. We also denote by $b_0\in G'(\Qps)$ the image of $b_0\in G(\Qps)$ under $i$. On the other hand, we write $i'$ for the identity map on $G'$, and set $\mu':=i^{\sharp}\circ \mu^{\sharp}$. 

\begin{prop}
\emph{The quintuple $(G^{\star},\mu^{\star},\bx^{\star},b_0,i^{\star})$ is an embedded Rapoport--Zink datum in the sense of \cite[Definition 4.1]{Oki2020b}. More precisely, we have the following: 
\begin{itemize}
\item $G^{\star}$ is a reductive connected group over $\Qp$,
\item $\mu^{\star} \colon \G_{m,\Qps}\rightarrow G^{\star}\otimes_{\Qp}\Qps$ is a minuscule cocharacter,
\item $\bx^{\star} \in \cB(G^{\star},\Qp)$,
\item $b_0^{\star}\in \bigcup_{w\in \Adm(\mu^{\star})}\cG_{\bx^{\star}}(\Zpb)\sigma(w)\cG_{\bx^{\star}}(\Zpb)$ is a decent element,
\item $i\colon G^{\star}\hookrightarrow G'$ is a closed immersion,
\item we have $i^{\star}\circ \mu^{\star}=\mu'$, $[i^{\star}(b_0^{\star})]\in B(G',\mu^{\star})$ and $i^{\star}(\bx^{\star})=\bx'$. 
\end{itemize}
Moreover, we have the following: 
\begin{enumerate}
\item $G^{\star}$ splits over a tamely ramified extension of $\Qp$ and $p\nmid \pi_1(G^{\star,\der})$,
\item there is an embedding of Shimura datum $\bi^{\star}\colon (\bG^{\star},\bX^{\star})\hookrightarrow (\bG',S_{\bC_0}^{\pm})$ which induces $i^{\star} \colon G^{\star} \hookrightarrow G'$. 
\item \cite[Axiom A]{Hamacher2019} holds for $(G^{\star},\mu^{\star},\bx^{\star},b_0,i^{\star})$ and a taken $\bi^{\star}$ as in (ii). 
\end{enumerate}}
\end{prop}

\begin{proof}
We only prove the assertions for $\star \in \{ \emptyset,\sharp \}$; otherwise it is clear. 

The assertion that $(G^{\star},\mu^{\star},\bx^{\star},b_0,i^{\star})$ is an embedded Rapoport--Zink datum is a consequence of Propositions \ref{gspi} (i), \ref{hsmn}, \ref{mucp}, \ref{hsbs} and $\ref{spbs}$. In the following, we check (i)--(iii). 

(i): This is a consequence of Corollary \ref{sptm} and Proposition \ref{gspi} (ii). 

(ii): This follows from Lemma \ref{sphs} and Proposition \ref{hssh}. 

(iii): This is contained in \cite[Proposition 7.8]{Zhou2020} since $[b_0]$ is basic. 
\end{proof}

Let $\Kb_p^{\star}:=\cG_{\bx^{\star}}(\Zpb)$, which is a subgroup of $G^{\star}(\Qpb)$. Put
\begin{equation*}
X_{\mu^{\star}}(b_0)_{\Kb_p^{\star}}(\Fpbar):=\left\{g\Kb_p^{\star}\in G^{\star}(\Qp)/\Kb_p^{\star}\,\middle| \, g^{-1}b_0 \sigma(g)\in \bigcup_{w\in \Adm(\mu^{\star})}\Kb_p^{\star} \sigma(w)\Kb_p^{\star} \right\}. 
\end{equation*}
As mentioned in \cite[\S 5.2]{Hamacher2019}, we can regard it as a perfect subscheme $X_{\mu^{\star}}(b^{\star})_{\Kb_p^{\star}}$ of the Witt vector affine Grassmanian of $\cG_{\bx^{\star}}$. Moreover, $i$ and $i^{\sharp}$ induce closed immersion of schemes
\begin{equation*}
X_{\mu}(b_0)_{\Kb_p}\hookrightarrow X_{\mu^{\sharp}}(b_0)_{\Kb_p^{\sharp}}\hookrightarrow X_{\mu'}(b_0)_{\Kb'_p}. 
\end{equation*}
We rephrase $X_{\sigma(\mu^{\star})}(b_0)_{\Kb_p^{\star}}(\Fpbar)$ as a set of quasi-isogenies of $p$-divisible groups. Let $(\X_0,\lambda_0,\tau_0)$ be the polarized $p$-divisible group together with isomorphism $\D(\X_0)\cong C^{*}\otimes_{\Zp}\Zpb$ as in Proposition \ref{expd}. Note that $\X_0$ is of dimension $2^{n}$ which is isoclinic of slope $1/2$. For a collection $(s^{\star}_{\alpha^{\star}})$ of $C^{\otimes}$ as in Definition \ref{tnsr} (if $\star=\,'$, we set $(s'_{\alpha'})$ as the empty set), put $t^{\star}_{\X_0,\alpha^{\star}}:=s^{\star}_{\alpha^{\star}}\otimes 1$. 

\begin{prop}\label{adrz}
\emph{There is an isomorphism between $X_{\sigma(\mu^{\star})}(b_0)_{\Kb_p^{\star}}(\Fpbar)$ and the set of isomorphism classes of triples $(X,\lambda,\rho)$, where
\begin{itemize}
\item $X$ is a $p$-divisible group of dimension $2^n$ and height $2^{n+1}$ over $\Fpbar$,
\item $\lambda \colon X\rightarrow X^{\vee}$ is a principal polarization, 
\item $\rho \colon \X_{0} \rightarrow X$ is a quasi-isogeny,
\end{itemize}
that satisfy the following: 
\begin{itemize}
\item $\rho^{\vee}\circ \lambda \circ \rho=c\lambda_{0}$ for some $c\in \Qpt$,
\item $\rho_{*}(t^{\star}_{\X,\alpha^{\star}})\in \D(X)^{\otimes}\otimes_{\Zpb}\Qpb$ lies in $\D(X)^{\otimes}$ for all $\alpha^{\star}$. 
\end{itemize}}
\end{prop}

\begin{proof}
This is a direct consequence of \cite[Definition/Lemma 5.5]{Hamacher2019}. 
\end{proof}

Take neat compact open subgroups $K^{\star,p}$ of $G(\A_{f}^{p})$ and ${K'}^p$ of $\bG'(\A_{f}^{p})$ satisfying the condition as in Section \ref{dfkp}. Let $\sS_{K'}$ be as in Section \ref{dfkp}, and $\sS_{K^{\star}}$ the Kisin--Pappas integral model of $\Sh_{K^{\star}}(\bG^{\star},\bX^{\star})$ associated to $\bi^{\star}$. 

On the other hand, we define a formal scheme $\RZ_{K'_p}$ over $\spf \Zpb$ as follows. Let $\nilp$ be the category of $\Zpb$-schemes in which $p$ is locally nilpotent. Then we define $\RZ_{K'_p}$ as the moduli space of the functor which parametrizes $S$-isomorphism classes of triples $(X,\lambda,\rho)$ for any $S\in \nilp$, where
\begin{itemize}
\item $X$ is a $p$-divisible group of dimension $2^n$ and height $2^{n+1}$ over $S$,
\item $\lambda \colon X\rightarrow X^{\vee}$ is a polarization,
\item $\rho \colon \X_{0}\otimes_{\Fpbar}\Sbar \rightarrow X\times_{S}\Sbar$ is a quasi-isogeny satisfying $\rho^{\vee}\circ \lambda \circ \rho=c(\rho)\lambda_{0}$ for some locally constant function $c(\rho)\colon \Sbar \rightarrow \Qpt$. 
\end{itemize}

Take $\widetilde{z}^{\star}=(z^{\star},j)$, where $z^{\star}\in \sSbar_{K^{\star}}^{\bs}(\Fpbar)$ and $j\colon \X_{0}\rightarrow \cA_{z_0}[p^{\infty}]$ is a quasi-isogeny which respects the polarizations up to $\Qpt$-multiple. Then we define a morphism of $\Zpb$-formal schemes
\begin{equation*}
\Theta'_{\widetilde{z}^{\star}} \colon \RZ_{K'_p}\rightarrow \sS_{K'},
\end{equation*}
as in \cite[\S 5.3]{Hamacher2019}, which is functorial with respect to $K^p$. Now consider the formal scheme
\begin{equation*}
\RZ_{K_p^{\star}}^{\diamond}:=\sS_{K^{\star}}\times_{\sS_{K'}}\RZ_{K'_p}, 
\end{equation*}
and denote by $\pi_1^{\star}\colon \RZ_{K_p^{\star}}^{\diamond}\rightarrow \sS_{K^{\star}}$ and $\pi_2^{\star}\colon \RZ_{K_p^{\star}}^{\diamond}\rightarrow \RZ_{K_p^{\star}}$ the first and second projections respectively. Moreover, let $(X^{\diamond},\lambda^{\diamond},\rho^{\diamond})$ be the pull-back of the universal object over $\spf \Zpb$ along $\pi_2^{\star}$. On the other hand, for $y\in \RZ_{K_p^{\star}}^{\diamond}(\Fpbar)$, we write $(X_y^{\diamond},\lambda_y^{\diamond},\rho_{y}^{\diamond})$ for the fiber of $(X^{\diamond},\lambda^{\diamond},\rho^{\diamond})$ at $y$, and put $t_{\alpha^{\star},y}^{\star,\diamond}:=\rho_{y}^{\diamond,-1}(t^{\star}_{\X_0,\alpha^{\star}})\in \D(X_y^{\diamond})^{\otimes}$. Note that the triple $(X_{y}^{\diamond},\lambda_{y}^{\diamond},(t_{\alpha^{\star},y}^{\star,\diamond}))$ satisfies Axiom \ref{exct} by Proposition \ref{adrz}. 

\begin{thm}\label{hkrz}(\cite[\S 5.3]{Hamacher2019})
\emph{There is a unique closed formal subscheme $\RZ_{K_{p}^{\star}}$ of $\RZ_{K_p^{\star}}^{\diamond}$ satisfying the following conditions:
\begin{enumerate}
\item $\RZ_{K_p^{\star}}^{\red,p^{-\infty}}\cong X_{\sigma(\mu^{\star})}(b_0)_{\Kb_p^{\star}}$,
\item $\widehat{\RZ}_{K_p^{\star},y}\cong \Def(X_{y}^{\diamond},\lambda_{y}^{\diamond},(t_{\alpha^{\star},y}^{\star,\diamond}))$ for any $y\in \RZ_{K_p^{\star}}(\Fpbar)$. 
\end{enumerate}
Moreover, the composite $\RZ_{K_p^{\star}}\hookrightarrow \RZ_{K_p^{\star}}^{\diamond} \xrightarrow{\pi_2^{\star}} \RZ_{K'_p}$ is a closed immersion. }
\end{thm}

We call $\RZ_{K_p^{\star}}$ as the \emph{Rapoport--Zink space} for $(G^{\star},\mu^{\star},\bx^{\star},b^{\star},i^{\star})$. We denote by
\begin{equation*}
\Theta_{\widetilde{z}^{\star}} \colon \RZ_{K^{\star}_p}\rightarrow \sS_{K^{\star}}
\end{equation*}
the composite $\RZ_{K^{\star}_p}\xrightarrow{i^{\star}} \RZ_{K^{\star}_p}^{\diamond}\xrightarrow{\pi_1^{\star}} \sS_{K'}$. 

\begin{rem}\label{fsmt}
\begin{enumerate}
\item The Rapoport--Zink space $\RZ_{K_p^{\star}}$ depends only on the quintuple by Theorem \ref{hkrz} (ii). Moreover, if $\star \in \{\sharp,\,'\}$, then it is further independent of $i^{\sharp}$ and $a_i \bmod (\Zpt)^2$ (it is appeared in the definition of $b_0$), and is formally smooth over $\spf \Zpb$. They follow from \cite[Theorem 4.9.1]{Kim2018a}. 
\item The embedded Rapoport--Zink datum $(G^{\sharp},\mu^{\sharp},\bx^{\sharp},b_0,\rho^{\sharp})$ corresponds to the local Shimura--Hodge datum $(\cG_{\bx^{\sharp}},b_0,\mu^{\sharp},C)$ in the sense of \cite[Definition 2.2.4]{Howard2017}. 
\end{enumerate}
\end{rem}

As in Remark \ref{fsmt}, $\RZ_{K_p^{\sharp}}$ and $\RZ_{K'_p}$ are formally smooth over $\spf \Zpb$. On the other hand, we have the following on $\RZ_{K_p}$:

\begin{thm}\label{rzrg}
\emph{
\begin{enumerate}
\item The Rapoport--Zink space $\RZ_{K_p}$ is regular of dimension $n-1$, and is flat over $\spf \Zpb$. The non-formally smooth locus is the set of $y\in \RZ_{K_p}(\Fpbar)$ such that $\rho_y \circ \iota_{X_y}(x_0)\circ \rho_y^{-1}$ acts by the zero map on $\Lie(X_y)$. 
\item If $y\in \RZ_{K_p}(\Fpbar)$ is a non-formally smooth point of $\RZ_{K_p}$, then we have $\widehat{\RZ}_{K_p,y}\cong \spf R_n$. 
\end{enumerate}}
\end{thm}

\begin{proof}
(i): By Proposition \ref{qtrc}, we obtain the regularity and the flatness of $\RZ_{K_p}$ and the equality $\dim (\RZ_{K_p})=n-1$. On the other hand, by \cite[Theorem 5.2 (ii)]{Oki2020b}, $y\in \RZ_{K_p}(\Fpbar)$ lies in the non-formally smooth locus if and only if $\iota_{X}(x_0)$ acts as the zero map on $\Fil^{1}(X)=(\Lie(X))^{*}$. Hence this is equivalent to the condition that $\iota_{X}(x_0)$ acts as the zero map on $\Lie(X)$. 

(ii): This is a consequence of Theorems \ref{hkrz} (ii) and \ref{qtrc} (ii). 
\end{proof}

\subsection{Group actions}

For $y \in \RZ_{K_p^{\star}}(S)$ where $S \in \nilp$, we denote by $(X_y^{\star},\lambda_y^{\star},\rho_y^{\star})$ the $p$-divisible group with quasi-isogeny corresponding to $y$. Moreover, we denote by $\iota_{X_y^{\star}}(x_0)$ be the quasi-isogeny of $X_y^{\star}$ induced by the tensor $x_0\in \End_{\Zp}(C^{*})$, which is an isogeny when $\star=\emptyset$. 

\begin{dfn}\label{defj}
Let $J^{\star}$ be the algebraic group over $\Qp$ defined to be
\begin{equation*}
J^{\star}(R)=\{g\in G^{\star}(R\otimes_{\Qp}\Qpb)\mid gb_0=b_0\sigma(g)\}
\end{equation*}
for any $\Qp$-algebra $R$. It is an inner form of $G^{\star}$ since $[b_0]$ is basic (\cite[5.2]{Kottwitz1985}). 
\end{dfn}
By definition, the closed immersion $i^{\star}\colon G^{\star}\hookrightarrow G'$ induces a closed immersion $j^{\star}\colon J^{\star}\hookrightarrow J'$. 

In the sequel, we equip $\RZ_{K^{\star}_p}$ with the action of $J^{\star}(\Qp)$ defined as in \cite[\S 5.4]{Hamacher2019}. Note that its action and the isomorphism in Theorem \ref{hkrz} (i) induce the canonical action on $X_{\mu^{\star}}(b_0)_{\Kb^{\star}_{p}}$. 

\begin{prop}\label{cije}
\emph{The closed immersion $i^{\star}_{*}\colon \RZ_{K_p^{\star}}\hookrightarrow \RZ_{K'_p}$ is compatible with the actions of $J^{\star}(\Qp)$ and $J'(\Qp)$, that is, we have $i^{\star}_{*}(g\cdot y)=j^{\star}(g)\cdot i^{\star}_{*}(y)$ for any $g\in J^{\star}(\Qp)$ and $y\in \RZ_{K_p^{\star}}$. }
\end{prop}

\begin{proof}
These follow from the definitions of $J^{\star}$ and the action of its $\Qp$-valued points on $\RZ_{K_p^{\star}}$. 
\end{proof}

\begin{prop}\label{embd}
\emph{The closed immersions $i\colon \RZ_{K_{p}}\hookrightarrow \RZ_{K'_{p}}$ and $i^{\sharp}\colon \RZ_{K_{p}^{\sharp}}\hookrightarrow \RZ_{K'_{p}}$ induce a closed immersion $\delta_{G,G^{\sharp}} \colon \RZ_{K_{p}}\hookrightarrow \RZ_{K_{p}^{\sharp}}$. Moreover, the following hold. 
\begin{enumerate}
\item $\delta_{G,G'}$ is compatible with the actions on $J(\Qp)$ and $J^{\sharp}(\Qp)$ in the same sense as Proposition \ref{cije}. 
\item The image of $\delta_{G,G^{\sharp}}$ is the locus of $y\in \RZ_{K_{p}^{\sharp}}$ where $\rho_y^{\sharp} \circ \iota_{X_y^{\sharp}}(x_0) \circ (\rho_y^{\sharp})^{-1}$ lifts to an isogeny of $X_y^{\sharp}$. 
\end{enumerate}}
\end{prop}

\begin{proof}
By the constructions of $\RZ_{K_{p}}$ and $\RZ_{K_{p}^{\sharp}}$, $i$ and $i^{\sharp}$ induce a morphism $\delta_{G,G^{\sharp}}\colon \RZ_{K_{p}}\rightarrow \RZ_{K_{p}^{\sharp}}$. It is a closed immersion by Theorem \ref{hkrz}. On the other hand, let $\cZ(x_{0})$ be the locus of $y\in \RZ_{K_{p}^{\sharp}}$ where $\rho_y^{\sharp} \circ \iota_{X_y^{\sharp}}(x_0)\circ (\rho_y^{\sharp})^{-1}$ lifts to an isogeny. It is a closed subscheme of $\RZ_{K_{p}^{\sharp}}$ by the rigidity of quasi-isogenies. See \cite[p.~52]{Rapoport1996}. 

In the sequel, we prove that $\delta_{G,G^{\sharp}}$ induces an isomorphism $\RZ_{K_p}\cong \cZ(x_0)$. First, we have
\begin{equation*}
\RZ_{K_p}^{\red,p^{-\infty}}=X_{\sigma(\mu)}(b_0)_{\Kb_p}^{\red,p^{-\infty}}
\end{equation*}
by Proposition \ref{adrz}. Moreover, Lemma \ref{admr} asserts the isomorphism
\begin{equation*}
X_{\sigma(\mu)}(b_0)_{\Kb_p}^{\red,p^{-\infty}}\cong \cZ(x_0)^{\red,p^{-\infty}}. 
\end{equation*}
Hence we obtain $\RZ_{K_p}^{\red,p^{-\infty}}\cong \cZ(x_0)^{\red,p^{-\infty}}$. On the other hand, the definitions of $\RZ_{K_{p}}$ and $\cZ(x_0)$ imply the existence of an isomorpshim $\widehat{\RZ}_{{K_p},y}\cong \widehat{\cZ}(x_0)_{y}$ for all $y\in \RZ_{K_p}(\Fpbar)$. Therefore we obtain the desired isomorphism. 
\end{proof}

Until the end of this section, we assume $\star=\emptyset$ for simplicity. We study the set of connected components $\pi_0(\RZ_{K_p})$ of $\RZ_{K_p}$. For $i\in \Z$, we define $\RZ_{K_p}^{(i)}$ as the locus of $y\in \RZ_{K_p}$ where $c(\rho_y(\Sbar))\subset p^{-i}\Zpt$. It is an open and closed formal subscheme of $\RZ_{K_p}$. 

\begin{thm}\label{jcis}
\emph{Let $g_0:=p^{-1}x_2x_1\in G(\Qpb)$. 
\begin{enumerate}
\item We have $g_0\in J(\Qp)$ and $g_0^2=-p^{-1}a_1a_2$. 
\item The element $g_0$ induces an isomorphism $\RZ_{K_p}^{(i)}\cong \RZ_{K_p}^{(i-1)}$ for $i\in \Z$. 
\item For any $i\in \Z$, $\RZ_{K_p}^{(i)}$ is connected. Hence
\begin{equation*}
\RZ_{K_p}=\coprod_{i\in \Z}\RZ_{K_p}^{(i)}
\end{equation*}
gives a decomposition into connected components. 
\end{enumerate}}
\end{thm}

\begin{proof}
(i), (ii): These follow from the definitions of $g_0$ and its action on $\RZ_{K_p^{\star}}$. 

(iii): First, note that we have $\pi_1(G)\cong \Z$ by Proposition \ref{kott} (i). Moreover, $\SO(V)$ is $\Qp$-simple, and the cocharacter $t\mapsto \mu(t)\bullet$ is non-central by Lemmas \ref{hsmn} and \ref{mucp} (ii). Hence \cite[Theorem 0.1 (2)]{He2020b} implies that the Kottwitz map $\widetilde{\kappa}_{G}$ induces a bijection
\begin{equation*}
\pi_0(X_{\sigma(\mu)}(b_0)_{\Kb_p})\cong \pi_1(G)_{I}^{\sigma}\cong \Z. 
\end{equation*}
On the other hand, Proposition \ref{kott} (ii) implies that the composite
\begin{equation*}
X_{\sigma(\mu)}(b_0)_{\Kb_p}\rightarrow \pi_0(X_{\sigma(\mu)}(b_0)_{\Kb_p}) \cong \Z,
\end{equation*}
is given by $\ord_{p}\circ \sml_{V}$. Now consider the composition of the above map with the isomorphism in Theorem \ref{hkrz} (i): 
\begin{equation*}
\RZ_{K_p}^{\red,p^{-\infty}}\cong X_{\sigma(\mu)}(b_0)_{\Kb_p}\rightarrow \Z. 
\end{equation*}
Then the fiber of $i\in \Z$ is equal to $\RZ_{K_p}^{(i),\red,p^{-\infty}}$. Therefore the family $\{\RZ_{K_p}^{(i)}\}_{i\in \Z}$ gives all connected components of $\RZ_{K_p}$. 
\end{proof}

\section{Deligne--Lusztig varieties for odd special orthogonal groups}\label{dlso}

In this section, we recall the results in \cite[Section 6]{Oki2019}, which involves Deligne--Lusztig varieties for odd and non-split even special orthogonal groups. The results will be used in Sections \ref{rzgs} and \ref{rznh}. 

Let $(\Omega_{d,0},[\,,\,])$ be a quadratic space of dimension $2d-1$ over $\Fp$. Put $G_{d,0}:=\SO(\Omega_{d,0})$, and we regard it as a subgroup of $\GL(\Omega_{d,0})\cong \GL_{2d-1,\Fp}$. Here this isomorphism is induced by a basis $e_1,\ldots,e_{2d-1}$ satisfying
\begin{equation*}
[e_i,e_j]=
\begin{cases}
2&\text{if }i=j=d,\\
1&\text{if }i+j=2d,i\neq j,\\
0&\text{otherwise}. 
\end{cases}
\end{equation*}
Moreover, let $B_{d,0}\subset G_{d,0}$ be the upper-half Borel subgroup, and $T_{d,0}\subset G_{d,0}$ the diagonal torus. Set $G_{d}:=G_{d,0}\otimes_{\Fp}\Fpbar$, $B_{d}:=B_{d,0}\otimes_{\Fp}\Fpbar$ and $T_{d}:=T_{d,0}\otimes_{\Fp}\Fpbar$. Consider the Weyl group $W_{d}:=N_{G_{d}}(T_{d})/T_{d}$, and we denote by $\Delta_{d}$ the set of simple reflections of $W_{d}$ associated to $(T_{d},B_{d})$. Then there is an isomorphism
\begin{equation*}
W_{d}\cong \{w\in \mathfrak{S}_{2d-1}\mid w(i)+w(2d-i)=2d\text{ for any }i\in \{1,\ldots,2d-1\}\}. 
\end{equation*}
Moreover, we have $\Delta_{d}=\{s_{d,1},\ldots,s_{d,d-2},t_{d}\}$, where
\begin{itemize}
\item $s_{d,i}=(i\ i+1)(2d-i-1\ 2d-i)$ for $i\in \{1,\ldots,d-2\}$,
\item $t_{d}=(d-2\ d)$. 
\end{itemize}
Now we set $w_{d}:=s_{d,d-2}\cdot \cdots \cdot s_{d,1}$. 

Let $\OGr(\Omega_{d,0})$ be the moduli space of maximally totally isotropic subspaces of $\Omega_{d,0}$, and put
\begin{equation*}
S_{d,0}:=\{L\in \OGr(\Omega_{d,0})\mid \rk(L\cap \sigma(L))\geq d-2\}. 
\end{equation*}
Furthermore, set $S_{d}:=S_{d,0}\otimes_{\Fp}\Fpbar$. 

\begin{prop}\label{oksf}(\cite[Proposition 6.4]{Oki2019})
\emph{There is a locally closed stratification
\begin{equation*}
S_{d}=\coprod_{i=0}^{d}X_{B_{i}}(w_i). 
\end{equation*}
Moreover, for $0\leq i\leq d$, the closure of $X_{B_{i}}(w_{i})$ in $S_{d}$ is equal to $\coprod_{j=0}^{i}X_{B_{i}}(x_{i})$. }
\end{prop}

We give a connection between $S_{d}$ and a projective variety relating with the non-split even special orthogonal group of dimension $2d$. Let $\Omega^{\sharp}_{d,0}$ be the non-split even special orthogonal group of dimension $2d$ over $\Fp$, and put $\Omega^{\sharp}_{d}:=\Omega^{\sharp}_{d,0}\otimes_{\Fp}\Fpbar$. We denote by $\OGr(\Omega^{\sharp}_{d})$ the moduli space of maximally totally isotropic subspaces of $\Omega^{\sharp}_{d}$, which is an $\Fpbar$-scheme. We have a decomposition of connected components
\begin{equation*}
\OGr(\Omega^{\sharp}_{d})=\OGr^{+}(\Omega^{\sharp}_{d})\sqcup \OGr^{-}(\Omega^{\sharp}_{d}),
\end{equation*}
and we have $\sigma(\OGr^{+}(\Omega^{\sharp}_{d}))=\OGr^{-}(\Omega^{\sharp}_{d})$. Moreover, put
\begin{gather*}
S^{\sharp}_{d}:=\{L^{\sharp}\in \OGr(\Omega^{\sharp}_{d})\mid \rk(L^{\sharp}\cap \sigma(L^{\sharp}))=d-1\},\quad S^{\sharp,\pm}_{d}:=S^{\sharp}_{d}\cap \OGr^{\pm}(\Omega^{\sharp}_{d}). 
\end{gather*}

\begin{prop}\label{sssm}(\cite[Proposition 5.3.2]{Howard2017})
\emph{The $\Fpbar$-scheme $S_d^{\sharp,\pm}$ is irreducible and projective smooth of dimension $d$. }
\end{prop}

\begin{prop}\label{odev}(\cite[Proposition 6.7]{Oki2019})
\emph{Let $\Omega^{\sharp}_{d,0}=\Omega_{d,0}\oplus l$ be an orthogonal decomposition of quadratic space. Then the morphism
\begin{equation*}
\OGr(\Omega^{\sharp}_{d})\rightarrow \OGr(\Omega_{d});L^{\sharp}\mapsto L^{\sharp}\cap \Omega_{d}
\end{equation*}
induces an isomorphism $S^{\sharp,\pm}_{d}\cong S_{d}$. }
\end{prop}

\section{Structure of the Rapoport--Zink space in the hyperspecial case}\label{rzgs}

In this section, we recall the results on $\RZ_{K_p^{\sharp}}$ by \cite[\S\S5--6]{Howard2017}. Let $(\X_0,\lambda_0,(t^{\sharp}_{0,\alpha^{\sharp}}),\tau_0)$ be a quadruple, where $(\X_0,\lambda_0,\tau_0)$ be a triple as in Proposition \ref{expd}, and $t^{\sharp}_{0,\alpha^{\sharp}}:=\tau_0(s^{\sharp}_{\alpha^{\sharp}}\otimes 1)$. Then the inclusion $\bLambda^{\sharp}\subset \End_{\Zp}(C^{*})$ as in Section \ref{quad} induces an inclusion $V^{\sharp}_{\Qpb}\subset \End_{\Qpb}(\D(\X_0)_{\Q})$. We define a $\sigma$-linear map $\D(\X_0)_{\Q}$ as $f\mapsto F\circ f\circ F^{-1}$. Then the isocrystal $(\End_{\Qpb}(\D(\X_0)_{\Q}),\Phi)$ is isoclinic of slope $0$ such that $V_{\Qpb}$ becomes a subisocrystal. 

Now we define a $\Qp$-vector space $\L_{0}^{\sharp}$ as the $\Phi$-fixed part of $V_{\Qpb}$. We regard $\L_0^{\sharp}$ as a quadratic space by $Q^{\sharp}$. Note that we can written as $Q^{\sharp}(v)=v\circ v$ under $\bLambda^{\sharp}\subset \End_{\Zp}(C^{*})$ by its definition. Moreover, Proposition \ref{hsbs} (i) implies $x_0\in \L_0^{\sharp}$. 

In the sequel, $y\in \RZ_{K_p^{\sharp}}$, we denote by $(X_y^{\sharp},\lambda_y^{\sharp},\rho_y^{\sharp})$ the polarized $p$-divisible group with quasi-isogeny associated to $y$. Moreover, for a $\Zp$-lattice $\Lambda$ in $\L_0^{\sharp}$, we write $\Lambda^{\cup}$ for the dual lattice of $\Lambda$ with respect to the above quadratic form.  

\begin{dfn}
\begin{enumerate}
\item A $\Zp$-lattice $\Lambda$ in $\L_0^{\sharp}$ is said to be a \emph{vertex lattice} in $\L_0^{\sharp}$ if $p\Lambda \subset \Lambda^{\cup}\subset \Lambda$. 
\item For a vertex lattice $\Lambda$ in $\L_0^{\sharp}$, the \emph{type} of $\Lambda$ is the integer $t(\Lambda):=\length_{\Zp}(\Lambda/\Lambda^{\cup})$. 
\end{enumerate}
\end{dfn}
We denote by $\VL(\L_0^{\sharp})$ the set of vertex lattices in $\L_0^{\sharp}$. Moreover, for $t\in \Znn$, set
\begin{equation*}
\VL(\L_0^{\sharp},t):=\{\Lambda \in \VL(\L_0^{\sharp})\mid t^{\sharp}(\Lambda)=t\}. 
\end{equation*}

\begin{prop}\label{vlhs}(\cite[Proposition 5.1.2]{Howard2017})
\emph{
\begin{enumerate}
\item The set $\VL(\L_0^{\sharp},t)$ is non-empty if and only if $t\in \{2,4,\ldots,t_{\max}^{\sharp}\}$, where
\begin{equation*}
t_{\max}^{\sharp}:=
\begin{cases}
n&\text{if $n$ is even},\\
n-1&\text{if $n$ is odd and $\disc(V^{\sharp})=(-1)^{(n+1)/2}$},\\
n+1&\text{if $n$ is odd and $\disc(V^{\sharp})\neq (-1)^{(n+1)/2}$}. 
\end{cases}
\end{equation*}
\item For $\Lambda \in \VL(\L_0^{\sharp})$, there is a vertex lattice in $\L_0^{\sharp}$ of type $t_{\max}^{\sharp}$ which contains $\Lambda$. 
\end{enumerate}}
\end{prop}

\begin{prop}\label{qtvl}
\emph{For $\Lambda \in \VL(\L_0^{\sharp})$, there is an isomorphism of the quadratic spaces
\begin{equation*}
(\Lambda/\Lambda^{\cup},pQ^{\sharp}\bmod p)\cong \Omega_{t^{\sharp}(\Lambda)/2,0}^{\sharp}. 
\end{equation*}}
\end{prop}

\begin{proof}
This is proved in \cite[5.3.1]{Howard2017}. 
\end{proof}

Now we attach a closed formal subscheme of $\RZ_{K_p^{\sharp}}$ for any vertex lattice in $\L_0^{\sharp}$. 
\begin{dfn}
For $\Lambda \in \VL(\L_0^{\sharp})$, we define $\RZ_{K_{p}^{\sharp},\Lambda}$ as the locus of $y\in \RZ_{K_{p}^{\sharp}}$ where
\begin{equation*}
\rho_y^{\sharp} \circ \Lambda \circ (\rho_y^{\sharp})^{-1}\subset \End(X_y^{\sharp}). 
\end{equation*}
It is a closed formal subscheme of $\RZ_{K_{p}^{\sharp}}$ by \cite[Proposition 2.9]{Rapoport1996}. Moreover, put
\begin{equation*}
\RZ_{K_p^{\sharp}}^{(0)}:=\RZ_{K_p^{\sharp},\Lambda}\cap \RZ_{K_p^{\sharp}}^{(0)}. 
\end{equation*}
\end{dfn}

Note that the action of $J^{\sharp}(\Qp)$ on $\RZ_{K_p^{\sharp}}$ and Proposition \ref{qtvl} induces an action of $\GSpin(\Omega_{t^{\sharp}(\Lambda)/2,0}^{\sharp})$ on $\RZ_{K_p^{\sharp},\Lambda}$. 

\begin{prop}\label{rdhs}(\cite[Theorem 4.2.11]{Li2018})
\emph{For $\Lambda \in \VL(\L_0^{\sharp})$, we have $\RZ_{K_p^{\sharp},\Lambda}=\RZ_{K_p^{\sharp},\Lambda}^{\red}$. }
\end{prop}

\begin{prop}\label{hprs}
\emph{Let $\Lambda \in \VL(\L_0^{\sharp})$. 
\begin{enumerate}
\item There is an isomorphism
\begin{equation*}
r_{\Lambda}^{\sharp}\colon p^{\Z}\backslash \RZ_{K^{\sharp}_p,\Lambda}\xrightarrow{\cong} S^{\sharp}_{t^{\sharp}(\Lambda)/2}, 
\end{equation*}
which commutes with the actions on $\GSpin(\Omega^{\sharp}_{t^{\sharp}(\Lambda)/2})$. 
\item The isomorphism $r_{\Lambda}^{\sharp}$ induces isomorphisms
\begin{equation*}
\RZ_{K^{\sharp}_p,\Lambda}^{(0)}\cong S^{\sharp,\pm}_{t^{\sharp}(\Lambda)/2}. 
\end{equation*}
In particular, $\RZ_{K^{\sharp}_p,\Lambda}^{(0)}$ is irreducible and projective smooth of dimension $(t^{\sharp}(\Lambda)/2)-1$ over $\Fpbar$. 
\end{enumerate}}
\end{prop}

\begin{proof}
(i): This is \cite[Theorem 6.3.1]{Howard2017}. 

(ii): The existence of isomorphisms is from \cite[Corollary 6.3.2]{Howard2017}. The rest of the assertion is a consequence of Proposition \ref{sssm}. 
\end{proof}

\begin{prop}\label{rzis}(\cite[Corollary 6.2.4]{Howard2017})
\emph{For $\Lambda_1,\Lambda_2\in \VL(\L_0^{\sharp})$, we have
\begin{equation*}
\RZ_{K_p^{\sharp},\Lambda_1}\cap \RZ_{K_p^{\sharp},\Lambda_2}=
\begin{cases}
\RZ_{K_p^{\sharp},\Lambda_1\cap \Lambda_2}&\text{if }\Lambda_1\cap \Lambda_2\in \VL(\L_0^{\sharp}),\\
\emptyset &\text{otherwise. }
\end{cases}
\end{equation*}}
\end{prop}

The following gives a description of the structure of $\RZ_{K_p^{\sharp}}$. 

\begin{thm}\label{bths}
\emph{There is an equality
\begin{equation*}
\RZ_{K_{p}^{\sharp}}=\bigcup_{\Lambda \in \VL(\L^{\sharp})}\RZ_{K_{p}^{\sharp},\Lambda}. 
\end{equation*}}
\end{thm}

\begin{proof}
This follows from \cite[Corollary 6.2.3]{Howard2017}. 
\end{proof}

We are going to give more information on the set of field valued points of $\RZ_{K_{p}^{\sharp}}$. Let $k/\Fpbar$ be a field extension, and denote by $W(k)$ the Cohen ring of $k$, that is, the discrete valuation ring with uniformizer $p$ whose residue field $k$. 
\begin{dfn}
A $W(k)$-lattice $L$ in $\L^{\sharp}_{0,W(k)}:=\L_0^{\sharp}\otimes_{\Zpb}W(k)$ is said to be a \emph{special lattice} if $L$ is self-dual and $\length_{W(k)}(L+\Phi(L)/L)=1$. 
\end{dfn}
We denote by $\SpL(\L_{0,W(k)}^{\sharp})$ the set of special lattices in $\L_{0,W(k)}^{\sharp}$. 

For $y\in \RZ_{K_{p}^{\sharp}}(k)$, let $M_y$ be the $W(k)$-lattice in $\D(\X_0^{\sharp})_{\Q}$ corresponding to $X_{y}^{\sharp}$, and set
\begin{equation*}
L_{y}:=\{v\in \L^{\sharp}_{0,W(k)}\mid v(F^{-1}(pM_{y}))\subset F^{-1}(pM_{y})\}. 
\end{equation*}

\begin{prop}\label{rphs}
\emph{
\begin{enumerate}
\item The $W(k)$-lattice $L_{y}$ is a special lattice in $\L_{0,W(k)}^{\sharp}$ for any $y\in \RZ_{K_{p}^{\sharp}}(k)$. 
\item The map $y\mapsto L_{y}$ induces a bijection $p^{\Z}\backslash \RZ_{K_{p}^{\sharp}}(k)\cong \SpL(\L^{\sharp}_{0,W(k)})$. 
\end{enumerate}}
\end{prop}

\begin{proof}
The assertions are contained in \cite[Proposition 6.2.2]{Howard2017}. 
\end{proof}

Finally, we give a relation between vertex lattices and special lattices. In particular, we rephrase Theorem \ref{bths} by means of special lattices. 

\begin{prop}\label{svhs}
\emph{
\begin{enumerate}
\item For $L\in \SpL(\L^{\sharp}_{0,W(k)})$, there is a minimum integer $d\in \{1,2,\ldots,t_{\max}^{\sharp}/2\}$ such that
\begin{equation*}
L^{(d)}:=L+\cdots +\Phi^{d}(L)
\end{equation*}
is $\Phi$-stable. Moreover, the $\Phi$-fixed part $\Lambda^{\sharp}(L)$ of $L^{(d)}$ is a vertex lattice in $\L_0^{\sharp}$ of type $2d$. 
\item For $\Lambda \in \VL(\L_0^{\sharp})$, the bijection $p^{\Z}\backslash \RZ_{K_{p}^{\sharp}}(k)\cong \SpL(\L^{\sharp}_{0,W(k)})$ induces a bijection
\begin{align*}
p^{\Z}\backslash \RZ_{K_{p}^{\sharp},\Lambda}(k) &\cong \{L\in \SpL(\L^{\sharp}_{0,W(k)})\mid \Lambda^{\sharp}(L)\subset \Lambda_{W(k)}\},\\
p^{\Z}\backslash \BT_{K_{p}^{\sharp},\Lambda}(k) &\cong \{L\in \SpL(\L^{\sharp}_{0,W(k)})\mid \Lambda^{\sharp}(L)=\Lambda \}. 
\end{align*}
\end{enumerate}}
\end{prop}

\begin{proof}
(i): This is \cite[Proposition 5.2.2]{Howard2017}. 

(ii): The assertion is contained in \cite[Proposition 6.2.2]{Howard2017}. 
\end{proof}

Finally, we give a description of the isomorphism in Proposition \ref{hprs} on $k$-valued points. By Propositions \ref{qtvl} and \ref{svhs} (ii), we obtain a bijection
\begin{equation*}
r_{\Lambda,k}^{\sharp}\colon p^{\Z}\backslash \RZ_{K_p^{\sharp},\Lambda}(k)\xrightarrow{\cong}S_{d}^{\sharp}(k);L\mapsto L/\Lambda^{\cup}_{W(k)}. 
\end{equation*}

\begin{prop}\label{rkhs}
\emph{The isomorphism $r_{\Lambda}^{\sharp}$ in Proposition \ref{hprs} induces $r_{\Lambda,k}^{\sharp}$ on $k$-valued points. }
\end{prop}

\begin{proof}
This follows from the construction of $r_{\Lambda}^{\sharp}$. See the proof of \cite[Theorem 6.3.1]{Howard2017}. 
\end{proof}

\section{Structure of the Rapoport--Zink space in the non-hyperspecial case}\label{rznh}

In this section, we construct the Bruhat--Tits stratification of the Rapoport--Zink space $\RZ_{K_p}$. 

\subsection{Vertex lattices}

We keep the notations in Section \ref{rzgs}. As mentioned in Section \ref{rzgs}, we have $x_0\in \L_0^{\sharp}$. We define a quadratic subspace $\L_0$ of $\L_0^{\sharp}$ as the perpendicular of $x_0$ in $\L_0^{\sharp}$. We denote by $\Lambda^{\vee}$ the dual lattice of $\Lambda$ in $\L_0$ for $\Lambda \in \VL(\L_0)$. Moreover, for $y\in \RZ_{K_p}$, $(X_y,\lambda_y,\rho_y)$ denotes the polarized $p$-divisible group with quasi-isogeny attached to $y$. 

\begin{dfn}
\begin{enumerate}
\item A $\Zp$-lattice $\Lambda$ in $\L_0$ is said to be a \emph{vertex lattice} in $\L_0$ if $p\Lambda \subset \Lambda^{\vee}\subset \Lambda$. 
\item We call $t(\Lambda):=\length_{\Zp}(\Lambda/\Lambda^{\vee})$ for the type of $\Lambda$. 
\end{enumerate}
\end{dfn}

We denote by $\VL(\L_0)$ the set of vertex lattices in $\L_0$. Furthermore, for $t\in \Znn$, set
\begin{equation*}
\VL(\L_0,t):=\{\Lambda \in \VL(\L_0)\mid t(\Lambda)=t\}. 
\end{equation*}

We are going to study a relation between $\VL(\L_0)$ and $\VL(\L_0^{\sharp})$. 

\begin{prop}\label{vtdl}
\emph{
\begin{enumerate}
\item For $\Lambda \in \VL(\L_0)$, the $\Zp$-lattice $\Lambda^{\sharp}:=\Lambda \oplus \Zp (p^{-1}x_0)$ is a vertex lattice in $\L_{0}^{\sharp}$. 
\item For $\Lambda \in \VL(\L_0)$, we have $t^{\sharp}(\Lambda^{\sharp})=t(\Lambda)+1$. 
\item The map $\Lambda \mapsto \Lambda^{\sharp}$ induces a bijection
\begin{equation*}
\VL(\L_0)\rightarrow \{\Lambda \in \VL(\L_0^{\sharp})\mid p^{-1}x_0\in \Lambda\}. 
\end{equation*}
The inverse map is given by $\Lambda' \mapsto \Lambda' \cap \L_0$ for $\Lambda' \in \VL(\L_0^{\sharp})$. 
\end{enumerate}}
\end{prop}

\begin{proof}
(i), (ii): This follows by definition. 

(iii): The injectivity is clear by definition. To prove the surjectivity, take $\Lambda' \in \VL(\L_0^{\sharp})$ containing $p^{-1}x_0$. It suffices to prove $\Lambda'=(\Lambda'\cap \L_0)\oplus \Zp(p^{-1}x_0)$. The inclusion $(\Lambda'\cap \L_0)\oplus \Zp(p^{-1}x_0)\subset \Lambda'$ is clear since $p^{-1}x_0\in \Lambda'$. For another inclusion, take $v'\in \Lambda'$, and write $v'=v+a(p^{-1}x_0)$ where $v\in \L_0$ and $a\in \Qp$. By $p\Lambda'\subset (\Lambda')^{\cup}$, we have $x_0\in (\Lambda')^{\cup}$. Hence we obtain
\begin{equation*}
-aa_1=[a(p^{-1}x_0),x_0]^{\sharp}=[v',x_0]^{\sharp}\in \Zp. 
\end{equation*}
Since $a_1\in \Zpt$, we have $a\in \Zp$. This implies $v\in \Lambda'$, which gives a desired assertion. 
\end{proof}

\begin{cor}\label{vltn}
\emph{
\begin{enumerate}
\item The set $\VL(\L_0,t)$ is non-empty if and only if $t\in \{1,3,\ldots,t_{\max}\}$, where 
\begin{equation*}
t_{\max}:=
\begin{cases}
n-1&\text{if $n$ is even},\\
n-2&\text{if $n$ is odd and $\varepsilon(V)=(p,-1)_{p}^{(n-1)/2}$},\\
n&\text{if $n$ is odd and $\varepsilon(V)\neq (p,-1)^{(n-1)/2}$}. 
\end{cases}
\end{equation*}
\item For $\Lambda \in \VL(\L_0)$, there is a vertex lattice in $\L_0$ of type $t_{\max}$ which contains $\Lambda$. 
\end{enumerate}}
\end{cor}

\begin{proof}
(i): By Propositions \ref{vtdl} (ii) and \ref{spsd} (iii), the conditions on $V$ to define $t_{\max}$ and those on $V^{\sharp}$ to define $t_{\max}^{\sharp}$ are equivalent respectively. The non-emptiness assertion is a consequence of Proposition \ref{vtdl} (iii). 

(ii): By Proposition \ref{vlhs}, there is $\Lambda'\in \VL(\L_0^{\sharp},t_{\max}^{\sharp})$ which contains $\Lambda^{\sharp}$. Then $\Lambda'\cap \L_0$ is a vertex lattice in $L_0$ containing $\Lambda$ by Proposition \ref{vtdl} (iii). 
\end{proof}

\begin{rem}
Corollary \ref{vltn} gives an information on the Hasse invariant of $\L_0$. More precisely, if $n$ is odd, we have $\varepsilon(\L_0)=(p,-1)_{p}^{(n-1)/2}$ if and only if $\varepsilon(V)\neq (p,-1)_{p}^{(n-1)/2}$. 
\end{rem}

\subsection{Bruhat--Tits stratification}

\begin{dfn}
For $\Lambda \in \VL(\L_0)$, we define a closed formal subscheme $\RZ_{K_p,\Lambda}$ of $\RZ_{K_p}$ as the locus of $y\in \RZ_{K_p}$ where $\rho_y \circ \Lambda \circ \rho_y^{-1}\subset \End(X_y)$. It is a closed formal subscheme of $\RZ_{K_p}$ by \cite[Proposition 2.9]{Rapoport1996}. 
\end{dfn}

\begin{prop}\label{btrs}
\emph{For $\Lambda \in \VL(\L_0)$, the closed immersion $\delta_{G,G^{\sharp}} \colon \RZ_{K_p}\hookrightarrow \RZ_{K_{p}^{\sharp}}$ induces an isomorphism $\RZ_{K_p,\Lambda}\cong \RZ_{K_p^{\sharp},\Lambda^{\sharp}}$. }
\end{prop}

\begin{proof}
This follows from the equality $\Lambda^{\sharp,\cup}=\Lambda^{\vee}\oplus \Zp x_0$ and the definitions of $\RZ_{K_p,\Lambda}$ and $\RZ_{K_p^{\sharp}}$. 
\end{proof}

\begin{cor}
\emph{For $\Lambda \in \VL(\L_0)$, we have $\RZ_{K_p,\Lambda}=\RZ_{K_p,\Lambda}^{\red}$. }
\end{cor}

\begin{proof}
We have $\RZ_{K_p,\Lambda}=\RZ_{K_p^{\sharp}}$ by Proposition \ref{btrs}, and the right-hand side is a reduced scheme by Proposition \ref{rdhs}. Hence the assertion follows. 
\end{proof}

\begin{cor}\label{rinh}
\emph{For $\Lambda_1,\Lambda_2\in \VL(\L_0)$, we have
\begin{equation*}
\RZ_{K_p,\Lambda_1}\cap \RZ_{K_p,\Lambda_2}=
\begin{cases}
\RZ_{K_p,\Lambda_1\cap \Lambda_2}&\text{if }\Lambda_1\cap \Lambda_2\in \VL(\L_0),\\
\emptyset &\text{otherwise. }
\end{cases}
\end{equation*}}
\end{cor}

\begin{proof}
This follows from Propositions \ref{btrs} and \ref{rzis}. 
\end{proof}

Let $\Lambda \in \VL(\L_0)$. By composing the isomorphisms in Propositions \ref{btrs}, \ref{hprs} and \ref{odev}, we obtain an isomorphism
\begin{equation*}
r_{\Lambda}\colon \RZ_{K_p,\Lambda}^{(0)}\xrightarrow{\cong}S_{(t(\Lambda)-1)/2}. 
\end{equation*}

\begin{cor}\label{psnh}
\emph{For $\Lambda \in \VL(\L_0)$, $\RZ_{K_p,\Lambda}^{(0)}$ is irreducible and projective smooth of dimension $(t(\Lambda)-1)/2$. }
\end{cor}

\begin{proof}
This follows from the isomorphism $r_{\Lambda}$ and Proposition \ref{sssm}. 
\end{proof}

\begin{dfn}
We say that a special lattice $L$ in $\L_{0,W(k)}^{\sharp}$ is \emph{$x_0$-special} if $x_0\in L$. 
\end{dfn}
We denote by $\SpL_{x_0}(\L_{0,W(k)}^{\sharp})$ the set of $x_0$-special lattices in $\L_{0,W(k)}^{\sharp}$. 

\begin{prop}\label{rpsm}
\emph{The bijection in Proposition \ref{rphs} (ii) induces a bijection
\begin{equation*}
p^{\Z}\backslash \RZ_{K_{p}}(k)\cong \SpL_{x_0}(\L_{0,W(k)}^{\sharp}). 
\end{equation*}}
\end{prop}

\begin{proof}
This follows from the Proposition \ref{embd} (ii). 
\end{proof}

\begin{prop}\label{svnh}
\emph{
\begin{enumerate}
\item For $L\in \SpL(\L_{0,W(k)})$, there is a minimum integer $d\in \{0,1,\ldots,(t_{\max}-1)/2\}$ such that
\begin{equation*}
L^{d}:=L+\cdots +\Phi^{d}(L)+\Zp(p^{-1}x_0)
\end{equation*}
is $\Phi$-stable. Moreover, the $\Phi$-fixed part $\Lambda(L)$ of $L^{d}\cap \L_{0}$ is a vertex lattice in $\L_0$ of type $2d+1$ such that $\Lambda^{\sharp}(L)\subset \Lambda(L)^{\sharp}$. 
\item For $\Lambda \in \VL(\L_0)$, the bijection $p^{\Z}\backslash \RZ_{K_{p}}(k)\cong \SpL(\L_{0,W(k)})$ induces a bijection 
\begin{align*}
p^{\Z}\backslash \RZ_{K_{p},\Lambda}(k) &\cong \{L\in \SpL_{x_0}(\L_{0,W(k)}^{\sharp})\mid \Lambda(L)\subset \Lambda \},\\
p^{\Z}\backslash \BT_{K_{p},\Lambda}(k) &\cong \{L\in \SpL_{x_0}(\L_{0,W(k)}^{\sharp})\mid \Lambda(L)=\Lambda \}. 
\end{align*}
\end{enumerate}}
\end{prop}

\begin{proof}
(i): Let $L\in \SpL_{x_0}(\L_0^{\sharp})$, and $d\in \Z$ as in Proposition \ref{svhs} (i). It suffices to prove the following. 
\begin{itemize}
\item If $p^{-1}x_0\in L^{(d)}$, then $d-1$ is the minimum integer among $m\in \Z$ such that $L^{m}$ is $\Phi$-stable, and we have $\Lambda(L)^{\sharp}=\Lambda^{\sharp}(L)$. 
\item If $p^{-1}x_0\in L^{(d)}$ (this implies $d<t_{\max}/2$), then $d$ is the minimum integer among $m\in \Z$ such that $L^{m}$ is $\Phi$-stable, and we have $\Lambda(L)^{\sharp}=\Lambda^{\sharp}(L)+\Zp(p^{-1}x_0)$. 
\end{itemize}
These follow from the same argument as \cite[Lemma 5.22]{Oki2019}. 

(ii): By Propositions \ref{btrs} and \ref{svhs} (ii), it suffices to prove that we have $\Lambda(L)\subset \Lambda$ if and only if $\Lambda(L)^{\sharp}\subset \Lambda^{\sharp}$ for a $x_0$-special lattice in $\L_{0,W(k)}$. However, this follows from (i), since we have $p^{-1}x_0\in \Lambda^{\sharp}$ by definition. 
\end{proof}

\begin{thm}\label{rzrd}
\emph{
\begin{enumerate}
\item There is an equality
\begin{equation*}
\RZ_{K_p}^{\red}=\bigcup_{\Lambda\in \VL(\L_0)}\RZ_{K_p,\Lambda}. 
\end{equation*}
\item Every irreducible component of $\RZ_{K_p}^{(0),\red}$ is of the form $\RZ_{K_p,\Lambda}^{(0)}$ for some $\Lambda \in \VL(\L_0,t_{\max})$. 
\end{enumerate}}
\end{thm}

\begin{proof}
(i): This is a consequence Proposition \ref{svnh}. 

(ii): By Corollaries \ref{vltn} (ii), \ref{rinh} and (i), we obtain
\begin{equation*}
\RZ_{K_p}^{\red}=\bigcup_{\Lambda\in \VL(\L_0,t_{\max})}\RZ_{K_p,\Lambda}. 
\end{equation*}
Hence the assertion follows by the above equality and Corollary \ref{psnh}. 
\end{proof}

Theorem \ref{rzrd} derives a locally closed stratification of $\RZ_{K_p}^{(0)}$. 

\begin{dfn}\label{btst}
For $\Lambda \in \VL(\L_0)$, we define a locally closed subscheme $\BT_{K_p,\Lambda}$ of $\RZ_{K_p}$ as follows: 
\begin{equation*}
\BT_{K_p,\Lambda}:=\RZ_{K_p,\Lambda}\setminus \left(\bigcup_{\Lambda'\subsetneq \Lambda}\RZ_{K_p,\Lambda'}\right). 
\end{equation*}
Moreover, put $\BT_{K_p,\Lambda}^{(0)}:=\BT_{K_p,\Lambda}\cap \RZ_{K_p}^{(0)}$. 
\end{dfn}

\begin{thm}\label{btnh}
\emph{
\begin{enumerate}
\item There is a locally closed stratification
\begin{equation*}
\RZ_{K_{p}}^{(0),\red}=\coprod_{\Lambda \in \VL(\L_0)}\BT_{K_{p},\Lambda}^{(0)}. 
\end{equation*}
\item For $\Lambda \in \VL(\L^{\sharp})$, we have an equality
\begin{equation*}
\RZ_{K_{p},\Lambda}^{(0),\red}=\coprod_{\Lambda' \subset \Lambda}\BT_{K_{p},\Lambda'}^{(0)}. 
\end{equation*}
\end{enumerate}}
\end{thm}

Next, we consider the non-formally smooth locus $\RZ_{K_{p}}^{\nfs}$ of $\RZ_{K_{p}}$. 

\begin{thm}\label{sgt1}
\emph{There is an equality
\begin{equation*}
\RZ_{K_{p}}^{\nfs}=\coprod_{\Lambda \in \VL(\L_{0},1)}\RZ_{K_p,\Lambda}=\coprod_{\Lambda \in \VL(\L_{0},1)}\BT_{K_p,\Lambda}. 
\end{equation*}}
\end{thm}

\begin{proof}
The proof is the same as that of \cite[Theorem 5.28]{Oki2019} by using Proposition \ref{svnh} (ii) and the following assertion. 
\begin{itemize}
\item For $x\in p^{\Z}\backslash \RZ_{K_p}(k)$ where $k$ is an algebraically closed field of characteristic $p$, we have $x\in \RZ_{K_{p}}^{\nfs}(k)$ if and only if $F^{-1}(pM_y)=x_0(M_y)$. 
\end{itemize}
Note that this follows from the proof of Proposition \ref{sgl0}. 
\end{proof}

Finally, we give a connection between Bruhat--Tits strata of $\RZ_{K_{p}}$ and Deligne--Lusztig varieties for odd special orthogonal groups. We use the notations in Section \ref{dlso}. 

\begin{thm}\label{dlnh}
\emph{For $\Lambda \in \VL(\L_0)$, the isomorphism $r_{\Lambda}\colon \RZ_{K_p}\xrightarrow{\cong} S_{(t(\Lambda)-1)/2}$ induces an isomorphism
\begin{equation*}
\BT_{K_p,\Lambda}^{(0)}\cong X_{B_{t(\Lambda)}}(w_{t(\Lambda)}). 
\end{equation*}}
\end{thm}

\begin{proof}
This follows from the same argument as \cite[Theorem 6.9]{Oki2019} by using the definition of $r_{\Lambda}$. 
\end{proof}

\section{Basic loci of Shimura varieties for spinor similitude groups}\label{bslc}

In this section, we consider the basic loci of the Kisin--Pappas integral model $\sS_{K}$ of the Shimura variety $\Sh_{K}(\bG,\bX)$, where $(\bG,\bX)$ is as in Section \ref{kpim}. We denote by $(\widehat{\sS}_{K,\Zpb})_{/\sSbar_{K}^{\bs}}$ the completion of $\sS_{K}\otimes_{\Z_{p}}\Zpb$ along the basic locus $\sSbar_{K}^{\bs}$. On the other hand, let $\bI$ be the unique inner form of $\bG$ which is anisotropic modulo center over $\R$ satisfying
\begin{equation*}
\bI \otimes_{\Q}\Ql \cong
\begin{cases}
\bG\otimes_{\Q}\Ql&\text{if }\ell \neq p,\\
J&\text{if }\ell=p. 
\end{cases}
\end{equation*}

\begin{thm}(\cite[Theorem 6.1]{Oki2020b})\label{unif}
\emph{There is an isomorphism of formal schemes
\begin{equation*}
\bI(\Q)\backslash (\RZ_{K_p}\times \bG(\A_f^p)/K^p)\cong (\widehat{\sS}_{K,\Zpb})_{/\sSbar_{K}^{\bs}}. 
\end{equation*}}
\end{thm}

\begin{thm}\label{ccl1}
\emph{
\begin{enumerate}
\item The basic locus $\sSbar_{K}^{\bs}$ is pure of dimension $(t_{\max}-1)/2$. Every irreducible component of $\sSbar_{K}^{\bs}$ is birational to a smooth projective variety. 
\item Let $\pi_0(\sSbar_{K}^{\bs})$ be the set of connected components of $\sSbar_{K}^{\bs}$. Then there is a bijection
\begin{equation*}
\pi_0(\sSbar_{K}^{\bs})\cong \bI(\Q)\backslash (\Z \times \bG(\A_f^p)/K^p). 
\end{equation*}
\item Let $\Irr(\sSbar_{K}^{\bs})$ be the set of irreducible components of $\sSbar_{K}^{\bs}$. Then there is a bijection
\begin{equation*}
\Irr(\sSbar_{K}^{\bs})\cong \bI(\Q)\backslash (\VL(\L_0,t_{\max})\times \bG(\A_f^p)/K^p). 
\end{equation*}
\end{enumerate}}
\end{thm}

\begin{proof}
(i): Take a complete representative $g_1,\ldots,g_m$ of $\bI(\Q)\backslash \bG(\A_f^p)/K^p$, and write
\begin{equation*}
\Gamma_{i}:=\bI(\Q)\cap (J(\Qp)\times g_iK^pg_i^{-1})\subset J(\Qp). 
\end{equation*}
Note that $\Gamma_i$ is discrete, torsion free and cocompact modulo center. Take $\Lambda \in \VL(\L_0,t_{\max})$. It suffices to prove that the canonical morphism $r_{i}\colon \RZ_{K_p}^{(0)}\rightarrow \Gamma_i \backslash \RZ_{K_p}$ induces an isomorphism $\BT_{K_p,\Lambda}^{(0)}\cong r_i(\BT_{K_p,\Lambda}^{(0)})$. By Corollary \ref{rinh}, it is reduced to prove $g\bullet \Lambda \neq \Lambda$ for $g\in J(\Qp)$ with $\sml_{V}(g)\in \Zpt$. However, this follows from the compactness of the stabilizer of $\Lambda$ in $\Qp$ and the fact that $\Gamma_i$ is discrete and torsion free. 

(ii): This follows from Theorem \ref{rzrd} (i) by taking connected components of the both-hand sides of the isomorphism in Theorem \ref{unif}. 

(iii): The proof is the same as (ii) by using Theorems \ref{unif} and \ref{rzrd} (ii). 
\end{proof}

\begin{thm}\label{sgrs}
\emph{There is a bijection $\sS_{K,\Zpb}^{\nsm}\cong \bI(\Q)\backslash (\VL(\L_0,1) \times \bG(\A_f^p)/K^p)$. }
\end{thm}

\begin{proof}
By Theorems \ref{sgbs} (iii), we have $\sS_{K}^{\nsm}\subset \sSbar_{K}^{\bs}$. On the other hand, Theorem \ref{unif} implies the isomorphism $\sS_{K}^{\nsm}\cong \bI(\Q)\backslash \RZ_{K_p}^{\nfs}\times \bG(\A_f^p)/K^p$. Hence the assertion follows from Theorem \ref{sgt1}. 
\end{proof}

\section{Arithmetic intersection on the Rapoport--Zink spaces}\label{aris}

In this section, we consider the GGP cycle associated to the closed immersion
\begin{equation*}
\delta_{G,G^{\sharp}} \colon \RZ_{K_p}\hookrightarrow \RZ_{K_p^{\sharp}}
\end{equation*}
as appeared in Proposition \ref{embd}. It induces a closed immersion
\begin{equation*}
(\id,\delta_{G,G^{\sharp}})\colon p^{\Z}\backslash \RZ_{K_p}\hookrightarrow p^{\Z}\backslash(\RZ_{K_p}\times_{\spf W} \RZ_{K_p^{\sharp}}). 
\end{equation*}

\begin{dfn}
We define the \emph{GGP cycle}, which is denoted by $\Delta_{G,G^{\sharp}}$, as the image of $\RZ_{K_p}$ under the closed immersion $(\id,\delta_{G,G^{\sharp}})$. 
\end{dfn}

For a subset $S$ of $\L_0^{\sharp}$, we define a closed formal subscheme $\cZ_0(S)$ of $\RZ_{K_p^{\sharp}}$ as the locus of $y\in p^{\Z}\backslash \RZ_{K_p^{\sharp}}$ where $(\rho_y^{\sharp})^{-1}\circ v\circ \rho_y^{\sharp}$ lifts to an isogeny of $X_y^{\sharp}$ for any $v\in S$. Note that we have $\cZ_0(x_0)=p^{\Z}\backslash \RZ_{K_p}$ by Proposition \ref{embd} (ii). 

\begin{prop}\label{trvn}
\emph{
\begin{enumerate}
\item The second projection $p^{\Z}\backslash (\RZ_{K_p}\times_{\spf \Zpb}\RZ_{K_p^{\sharp}})\rightarrow p^{\Z}\backslash \RZ_{K_p^{\sharp}}$ induces an isomorphism
\begin{equation*}
\Delta_{G,G^{\sharp}}\cap g\Delta_{G,G^{\sharp}} \cong \cZ_0(L_{x_0}(g))^{g}. 
\end{equation*}
\item If the formal scheme $\cZ_0(L_{x_0}(g))^g$ is an artinian scheme, then we have
\begin{equation*}
\O_{\Delta_{G,G^{\sharp}}} \otimes^{\bL}\O_{g\Delta_{G,G^{\sharp}}}=\O_{\Delta_{G,G^{\sharp}}} \otimes \O_{g\Delta_{G,G^{\sharp}}},
\end{equation*}
that is, the tensor product of sheaves $\O_{\Delta_{G,G^{\sharp}}} \otimes \O_{\Delta_{G,G^{\sharp}}}$ represents the derived tensor product $\O_{\Delta_{G,G^{\sharp}}} \otimes^{\bL}\O_{\Delta_{G,G^{\sharp}}}$ in the category of sheaves on $p^{\Z}\backslash (\RZ_{K_p}\times_{\spf \Zpb}\RZ_{K_p^{\sharp}})$. 
\end{enumerate}}
\end{prop}

\begin{proof}
The proof is the same as that of \cite[Proposition 8.4]{Oki2019}. 
\end{proof}

\begin{dfn}
Let $g\in J^{\star}(\Qp)$ in which $\cZ(L_{x_0}(g))^{g}$ is an artinian scheme. Then we set
\begin{equation*}
\langle \Delta_{G,G^{\sharp}},g\Delta_{G,G^{\sharp}}\rangle:=\chi(\O_{\Delta_{G,G^{\sharp}}} \otimes^{\bL}\O_{g\Delta_{G,G^{\sharp}}}),
\end{equation*}
where $\chi$ is the Euler-Poincar{\'e} characteristic. 
\end{dfn}

In the sequel, we study $\langle \Delta_{G,G^{\sharp}},g\Delta_{G,G^{\sharp}}\rangle$ in a special case. 

\begin{dfn}
We say that $g\in J^{\sharp}(\Qp)$ is \emph{regular semisimple minuscule} if $L_{x_0}(g)$ is a vertex lattice in $\L^{\sharp}$. 
\end{dfn}

If $g$ is regular semisimple minuscule, then $g\bullet L_{x_0}(g)=L_{x_0}$. Hence $g$ induces an action on $\Omega^{\sharp}(g):=L_{x_0}(g)^{\cup}/L_{x_0}(g)$. We denote by $\gbar \in \GL(\Omega^{\sharp}(g))(\Fp)$ the element induced by $g$, and by $P_g\in \Fp[T]$ the characteristic polynomial of $\gbar$. 

\begin{dfn}
For a monic polynomial $R\in \Fp[T]$, set $R^{*}(T):=T^{\deg(R)}R(T^{-1})$. We say that $R$ is self-reciprocal if $R^{*}=R$. 
\end{dfn}

\begin{lem}\label{pgsr}
\emph{The element $\gbar \in \GL(\Omega^{\sharp}(g))(\Fp)$ lies in $\SO(\Omega^{\sharp}(g))(\Fpbar)$. In particular, $P_g$ is self-reciprocal. }
\end{lem}

\begin{proof}
This follows from the same argument as \cite[Lemma 5.2.2]{He2019}. 
\end{proof}

We denote by $\SR(P_g)$ and $\NSR(P_g)$ the sets of irreducible monic $R$ satisfying $R^{*}=R$ and $R^{*}\neq R$ respectively. Moreover, let $\NSR(P_g)$ be the quotient of $\NSR(P_g)$ by the equivalence relation $R^{*}=R$. 

For $R\in \SR(P_g)\sqcup \NSR(P_g)$, we define $m_{R}(P_g)\in \Znn$ as the multiplicity of $R$ in $P_g$. Note that we have $m_{R}(P_g)=m_{R^{*}}(P_g)$ by the definition of $R^{*}$. 

\begin{thm}\label{itar}
\emph{Assume that $g\in J^{\sharp}(\Qp)$ is regular semisimple minuscule and $\RZ_{K_{p}^{\sharp}}^{g}$ is non-empty. 
\begin{enumerate}
\item We have $\Delta_{G,G^{\sharp}} \cap g\Delta_{G,G^{\sharp}} \neq \emptyset$ if and only if there is a unique $Q_g\in \SR(P_g)$ such that $m_{Q_g}(P_g)$ is odd. 
\item If the equivalent conditions in (i) hold, then the cardinality of $(\Delta_{G,G^{\sharp}} \cap g\Delta_{G,G^{\sharp}})(\Fpbar)$ equals
\begin{equation*}
\deg(Q_g)\cdot \prod_{[R]\in \NSR(P_g)}(1+m_{R}(P_g)). 
\end{equation*}
\end{enumerate}}
\end{thm}

\begin{proof}
The proof is the same as \cite[Theorem 3.6.4]{Li2018}. 
\end{proof}

\begin{prop}\label{isrd}
\emph{For $g\in J^{\sharp}(\Qp)$ is regular semisimple minuscule, then the formal scheme $\Delta_{G,G^{\sharp}} \cap g\Delta_{G,G^{\sharp}}$ is an artinian $\Fpbar$-scheme. }
\end{prop}

\begin{proof}
We have $\cZ_0(L_{x_0}(g))=p^{\Z}\backslash \RZ_{K_p,L_{x_0}(g)}$ by definition, which is a reduced scheme by Proposition \ref{rdhs}. Hence $\Delta_{G,G^{\sharp}} \cap g\Delta_{G,G^{\sharp}}$ is an $\Fpbar$-scheme Proposition \ref{trvn} (i). The rest of the assertion is a consequence of the result as proved above and Theorem \ref{itar} (ii). 
\end{proof}

By Propositions \ref{isrd} and \ref{trvn} (ii), it turns out that the intersection multiplicity $\langle \Delta_{G,G^{\sharp}},g\Delta_{G,G^{\sharp}}\rangle$ is defined. We give an expression of $\langle \Delta_{G,G^{\sharp}},g\Delta_{G,G^{\sharp}}\rangle$ by means of the polynomial $P_g$. 

\begin{thm}\label{thai}
\emph{Let $g\in J^{\sharp}(\Qp)$ be regular semisimple minuscule. Assume that $\Delta_{G,G^{\sharp}} \cap g\Delta_{G,G^{\sharp}}$ is non-empty. Then we have
\begin{equation*}
\langle \Delta_{G,G^{\sharp}},g\Delta_{G,G^{\sharp}} \rangle =\deg(Q_g)\cdot \frac{m_{Q_g}(P_g)+1}{2}\prod_{[R]\in \NSR(P_g)/\sim}(1+m_{R}(P_g)),
\end{equation*}
where $Q_g\in \SR(P_g)$ is as in Theorem \ref{itar} (i). }
\end{thm}

\begin{proof}
The proof is the same as that of \cite[Theorem 5.2.4 (3)]{He2019} by using Lemma \ref{pgsr} and \cite[Theorem 4.3.3 (3)]{He2019}. 
\end{proof}

\end{document}